\documentclass[reqno]{amsart}

\usepackage[applemac]{inputenc}

\usepackage{amssymb}
\usepackage{graphicx}
\usepackage[cmtip,all]{xy}
\usepackage{verbatim}
\usepackage{tikz-cd}
\usepackage{mathtools}
\usepackage{graphicx}
\usepackage{mathrsfs}
\usepackage{geometry}
\usepackage{hyperref}
\usepackage{enumitem}
\DeclareMathAlphabet{\mathpzc}{OT1}{pzc}{m}{it}

\newtheorem{theorem}{Theorem}[section]
\newtheorem{lemma}[theorem]{Lemma}
\newtheorem{corollary}{Corollary}[section]
\newtheorem{problem}[theorem]{Problem}

\theoremstyle{definition}
\newtheorem{definition}[theorem]{Definition}
\newtheorem{example}[theorem]{Example}

\theoremstyle{remark}
\newtheorem{remark}[theorem]{Remark}

\numberwithin{equation}{section}

\makeatletter
\DeclareRobustCommand{\cev}[1]{%
  \mathpalette\do@cev{#1}%
}
\newcommand{\do@cev}[2]{%
  \fix@cev{#1}{+}%
  \reflectbox{$\m@th#1\vec{\reflectbox{$\fix@cev{#1}{-}\m@th#1#2\fix@cev{#1}{+}$}}$}%
  \fix@cev{#1}{-}%
}
\newcommand{\fix@cev}[2]{%
  \ifx#1\displaystyle
    \mkern#23mu
  \else
    \ifx#1\textstyle
      \mkern#23mu
    \else
      \ifx#1\scriptstyle
        \mkern#22mu
      \else
        \mkern#22mu
      \fi
    \fi
  \fi
}

\makeatother

 \newcommand{\virgolette}{``}

\newcommand*{\defeq}{\mathrel{\vcenter{\baselineskip0.5ex \lineskiplimit0pt
                     \hbox{\scriptsize.}\hbox{\scriptsize.}}}%
                     =}

\newcommand\asim{\mathrel{%
  \ooalign{\raise0.1ex\hbox{$\sim$}\cr\hidewidth\raise-0.8ex\hbox{\scalebox{0.9}{$\scriptstyle{x}$}}\hidewidth\cr}}}

\newcommand{\mani}{\ensuremath{\mathpzc{M}}}
\newcommand{\manir}{\ensuremath{\mathpzc{M}_{red}}}
\newcommand{\stsheaf}{\ensuremath{\mathcal{O}_{\mathpzc{M}}}}

\newcommand{\K}{\mathcal{K}}

\newcommand{\beq}{\begin{equation}}
\newcommand{\eeq}{\end{equation}}
\newcommand{\bear}{\begin{eqnarray}}
\newcommand{\eear}{\end{eqnarray}}

\begin{document}

\title{The Universal de Rham/Spencer Double Complex on a Supermanifold}

\author{Sergio L. Cacciatori}
\address{Universit\`a degli Studi dell'Insubria and INFN - Sezione di Milano}
\curraddr{Via Valleggio 11, 22100, Como, Italy and Via Celoria 16, 20133, Milano, Italy}
\email{sergio.cacciatori@uninsubria.it}

\author{Simone Noja}
\address{Universit\"at Heidelberg}
\curraddr{Im Neuenheimer Feld 205, 69120, Heidelberg, Germany}
\email{noja@mathi.uni-heidelberg.de}

\author{Riccardo Re}
\address{Universit\`a degli Studi dell'Insubria}
\curraddr{Via Valleggio 11, 22100, Como, Italy}
\email{riccardo.re@uninsubria.it}

%\subjclass[2020]{Primary 32C11}
%    The 2010 edition of the Mathematics Subject Classification is
%    now available.  If you are citing a classification from the
%    new scheme, use the following input coding instead.


\begin{abstract} The universal Spencer and de Rham complexes of sheaves over a smooth or analytical manifold are well known to play a basic role in the theory of $\mathcal{D}$-modules. In this article we consider a double complex of sheaves  generalizing both complexes for an arbitrary supermanifold, and we use it  to unify  the notions of differential and integral forms on real, complex and algebraic supermanifolds. 
The associated spectral sequences give the de Rham complex of differential forms and the complex of integral forms at page one. For real and complex 
supermanifolds both spectral sequences converge at page two to the locally constant sheaf. We use this fact to show that the cohomology of differential forms is isomorphic to the cohomology of integral forms, and they both compute the de Rham cohomology 
of the reduced manifold. Furthermore, we show that, in contrast with the case of ordinary complex manifolds, the Hodge-to-de Rham (or Fr\"olicher) spectral sequence of supermanifolds with K\"ahler reduced manifold does not converge in general at page one. 
\end{abstract}

\maketitle

\tableofcontents

\section{Introduction}

\noindent The mathematical theory of forms on supermanifolds, together with the related integration theory, is one of the most peculiar and non-trivial aspect of supergeometry \cite{Deligne} \cite{Manin} \cite{Witten}. Indeed, whereas on an ordinary manifold differential forms anticommute, so that the de Rham complex terminates at the dimension of the manifold, in supergeometry the differentials of odd functions do commute instead. This apparently trivial fact has far-reaching consequences, namely it implies that the de Rham complex of a supermanifold is not bounded from above \cite{Manin} \cite{Penkov}. This, in turn, leads to the failure of Poincaré duality, as there is no notion of a {top} differential form which yields a tensor density that can be integrated over a supermanifold. In order to cure this pathology and define a meaningful integration theory analogous to the ordinary integration of differential forms in the classical setting, the notion of \emph{integral form} has been developed and introduced in supergeometry \cite{Deligne} \cite{Manin} \cite{Penkov}. Integral forms fit into a complex which, in some sense, is dual to the de Rham complex of differential forms: whereas the de Rham complex of a supermanifold is not bounded from above, the complex of integral forms - or Spencer complex of a supermanifold - is not bounded from below. In particular, one of the most peculiar and defining supergeometric construction, that of Berezinian sheaf \cite{Ruiperez} \cite{NojaRe} - whose sections can be integrated over the supermanifold - plays the role of the top sheaf in the complex of integral forms, thus providing a substitute for the notion of canonical sheaf in supergeometry. \\
%\noindent The notion of \emph{integral form} has been developed in supergeometry for the purpose of defining an integration theory analogous to the ordinary integration of differential forms on (sub)manifolds in the classical setting. 
\noindent In this view, integral forms appear in supergeometry as more useful and natural mathematical objects than differential forms. However, on the other hand, the definitions of sheaves of differential forms, 
vector fields or also sheaves of linear differential operators are easily available in supergeometry by the same constructions as in classical geometry: consider for example the construction of these objects due to Grothendieck, which applies to an extremely
general setting. 
One purpose of this paper is to give a new construction of integral forms which is both coordinate-free and built upon the more standard notions of differential forms and operators, with the future aim of studying possible generalizations to other 
classes of ``superforms". \\
We stress that the \emph{syntax} of such objects is known in the physics literature: this means that there exists a formalism of integral forms expressed in terms of coordinates, together with associated calculus and transformation rules 
\cite{Witten} \cite{WittenSuper}.   
This formalism has been further expanded to an extended formalism of ``superforms", generalization of integral forms, which have been developed and applied, for example, in the recent \cite{CCG} \cite{CCGN} \cite{CGNinf} \cite{CGN}  \cite{CA1} \cite{CGP}.
On the \emph{semantic} side, of course a coordinate-free construction of integral forms exists in the supergeometry literature, see \cite{Manin}, but in a way which is unrelated to differential forms, so that there is no obvious relations between these two concepts. \\
\noindent
In the present paper, starting from first principles, we unify the notions of differential and integral forms and their related complexes on real, complex and also algebraic supermanifolds. In particular, given the natural sheaves of differential operators 
$\mathcal{D}_\mani$ and differential forms ${\Omega}^{\bullet}_{\mani, odd}$ on a certain supermanifold $\mani$, we start from the so-called \emph{universal de Rham complex} $\Omega^{\bullet}_{\mani, odd} \otimes_{\stsheaf} \mathcal{D_\mani}$ and 
\emph{universal Spencer complex} $\mathcal{D}_{\mani} \otimes_{\stsheaf} (\Omega^{\bullet}_{\mani, odd})^\ast$ and we show that they can be unified into a single double complex of sheaves supported on the triple tensor product 
$\Omega^{\bullet}_{\mani, odd} \otimes_{\stsheaf} \mathcal{D}_\mani \otimes_{\stsheaf} (\Omega^{\bullet}_{\mani, odd})^{\ast}$, which we call the de Rham/Spencer double complex. This displays a truly non-commutative behavior, rather than just a super-commutative one, due to the presence of the sheaf $\mathcal{D}_\mani$ in 
the pivotal position. The two spectral sequences associated to this double complex yield, at page one, the complex of differential forms and the complex of integral forms on $\mani$. Furthermore, in the case of real or complex supermanifolds, 
both spectral sequences converge at page two to the sheaf of locally constant functions over $\mathbb{R}$ or $\mathbb{C}$, depending on the supermanifold being real or complex. This is a consequence of the \emph{Poincar\'e lemmas} for differential and 
integral forms. More precisely, we prove the following Theorem, which gathers Theorems \ref{UniversalDR} and \ref{USC} from Sections 3 and 4 respectively and Theorems \ref{doublecomplex}, \ref{diffform} and  \ref{intform} from Section 5 of the paper.
\begin{theorem}[Main Theorem] Let $(\mani, \mathcal{O}_\mani)$ be a real, complex of algebraic supermanifold. Then 
\begin{enumerate}[leftmargin=*]
\item the homology of the universal de Rham complex $(\Omega^{\bullet}_{\mani, odd} \otimes_{\stsheaf} \mathcal{D_\mani}, D)$ is naturally isomorphic to the Berezinian sheaf $\mathcal{B}er (\mani);$
\item the homology of the universal Spencer complex  $(\mathcal{D}_{\mani} \otimes_{\stsheaf} (\Omega^{\bullet}_{\mani, odd})^\ast, \delta )$ is naturally isomorphic to the structure sheaf $\mathcal{O}_\mani$.
\end{enumerate}
The universal de Rham complex and the universal Spencer complex can be unified into a double complex $(_{\mathpzc{D}} \mathcal{V}_{\mani}^{\bullet \bullet}, \hat D, \hat \delta )$ with associated spectral sequences $(E_r^{\Omega}, d^\Omega_r)$ and $(E_r^{\Sigma}, d^\Sigma_r)$. Then 
\begin{enumerate}[leftmargin=*]
\item $(E_1^\Omega, d_1^\Omega) $ is isomorphic to the complex of differential forms on $\mani$ and $(E_1^\Sigma, \delta^\Sigma_1)$ is isomorphic to the complex of integral forms on $\mani$;
\item Provided that $\mani$ is a real or a complex supermanifold, both of the spectral sequences converge at page 2 to the constant sheaf valued in the real or in the complex numbers.
\end{enumerate}
%the universal de Rham complex $(\Omega^{\bullet}_{\mani, odd} \otimes_{\stsheaf} \mathcal{D_\mani}, D)$ and the universal Spencer complex $(\mathcal{D}_{\mani} \otimes_{\stsheaf} (\Omega^{\bullet}_{\mani, odd})^\ast, \delta )$ can be unified into a bicomplex $_{\mathpzc{D}} \mathcal{V}_{\mani}^{\bullet \bullet} $
\end{theorem}

\noindent Whilst the proof of the Poincar\'e lemma for differential forms in a supergeometric context is well-known and it consists of a straightforward generalization of the ordinary one, the Poincar\'e lemma for integral forms, instead, is a hallmark of 
supergeometry and we will prove it in detail in Theorem \ref{PLInt}. 
Finally, we enhance the above de Rham/Spencer double complex of sheaves with a triple complex structure, by taking its \v{C}ech cochains, see Definition \ref{triplecompl}. In this way one obtains two double complexes: one is the \emph{\v{C}ech-de Rham double complex} of differential forms on $\mani$ 
and the other is the \emph{\v{C}ech-Spencer double complex} of integral forms on $\mani$. We show that the related spectral sequences both converge to the \emph{de Rham cohomology} of the reduced manifold, showing that the cohomology of differential forms is isomorphic to the cohomology of integral forms, and, in turn, that the complexes of differential and integral forms are quasi-isomorphic. More precisely, we prove from first principles Theorem \ref{Cechde}, which we phrase here as follows.
\begin{theorem}[Quasi-Isomorphism] Let $\mani$ be a real supermanifold. The cohomology of differential forms $H^\bullet_{\mathpzc{dR}} (\mani)$ and the cohomology of integral forms 
$H^\bullet_{\mathpzc{Sp}} (\mani) $ are isomorphic, \emph{i.e.}\
\bear
H^\bullet_{\mathpzc{dR}} (\mani) \cong \check{H}^\bullet (\mani, \mathbb{R}_\mani) \cong H^\bullet_{\mathpzc{Sp}} (\mani), \nonumber
\eear 
In particular the complex of differential forms and of integral forms on $\mani$ are quasi-isomorphic.
\end{theorem}
%proving the equivalence of the cohomology of differential and integral forms. 
\noindent Nonetheless, we remark that in the case of a complex supermanifold with K\"ahler reduced manifold, something intriguing happens. Indeed, the Hodge-to-de Rham (or Fr\"olicher) spectral sequence, which still computes the de Rham cohomology of the reduced manifold, does {not} converge 
at page one, as it does in the ordinary setting. 
Instead, there are many more non-trivial maps, as shown in example \ref{example}, thus hinting at new promising developments in the geometry of complex supermanifolds.

\

\noindent {\bf Acknowledgments.} The authors wish to thank Ivan Penkov for fruitful discussions and advice. %and a very pleasant tour of  the italian city of Cremona. 

\section{Setting the Stage: Main Definitions}

\noindent In the following we will work over a real, complex analytic or algebraic supermanifold $\mani$ unless otherwise stated. See the classical textbook \cite{Manin} for a thorough introduction, or the recent \cite{P2} by the authors for a short 
compendium to the topic. \\
We let $\mani$ be a supermanifold of dimension $p|q$ and we denote its reduced space, which is an ordinary (real, complex or algebraic) manifold of dimension $p$, by $\manir$. In particular, we will deal with the sheaf of $1$-forms 
$\Omega^{1}_{\mani, odd}$ on $\mani$. This is a locally-free sheaf on $\mani$ of rank $q|p$. Indeed, if we let $U$ be an open set in the topological space underlying $\manir$ and we set $x_a  =  z_i | \theta_\alpha $ for $i=1, \ldots, p$ and 
$\alpha = 1, \ldots,q$ to be a system of local coordinates over $U$ for the supermanifold $\mani$, we have that 
\bear
\Omega^{1}_{\mani, odd} (U) = \{ d\theta_1, \ldots, d\theta_q | dz_1, \ldots, dz_p \} \cdot \mathcal{O}_\mani (U) ,
\eear
where $\mathcal{O}_\mani$ is the structure sheaf of $\mani$. We stress that the $d\theta$'s are {even} and the $dz$'s are {odd}, as we take the differential $d: \stsheaf \rightarrow \Omega^1_{\mani, odd}$ to be an {odd} morphism. 
Also, note 
that we have written $\Omega^1_{\mani, odd}$ as a locally-free sheaf of {right} $\mathcal{O}_\mani$-modules.\\
Likewise, we denote the dual of $\Omega^{1}_{\mani, odd} $ with $(\Omega_{\mani, odd}^1)^\ast$. This can be canonically identified with the sheaf $\Pi \mathcal{T}_\mani$, where $\mathcal{T}_\mani$ is the tangent sheaf of $\mani$ and $\Pi$ - 
the so-called parity-changing functor - is there to remind that the parity of the sheaf is reversed, so that the rank of $\Pi \mathcal{T}_\mani$ is actually $q|p$. We will call a section of $\Pi \mathcal{T}_\mani = (\Omega^1_{\mani, odd})^\ast$ a $\Pi$-vector 
field or vector field for short. Locally, 
$(\Omega^1_{\mani, odd})^\ast$ is generated by expressions of the kind
\bear
(\Omega^{1}_{\mani, odd})^\ast (U) = \mathcal{O}_{\mani} (U) \cdot \{ \pi \partial_{\theta_1}, \ldots, \pi \partial_{\theta_q} | \pi \partial_{z_1}, \ldots, \pi \partial_{z_p} \},
\eear
where the $\pi \partial_{\theta}$'s are even and the $\pi \partial_{z}$'s are odd. Notice that $(\Omega^1_{\mani, odd})^\ast$ has been written with the structure of locally-free sheaf of {left} $\mathcal{O}_\mani$-modules. 

\

\noindent Applying the supersymmetric power functor $S^\bullet$ to the sheaf $\Omega^1_{\mani, odd}$ and $(\Omega^1_{\mani, odd})^\ast$ one gets the usual notion of (differentially graded) algebra of forms and polyfields over a supermanifold. 
In particular, we call a section 
of the sheaf $\Omega^{k}_{\mani, odd} \defeq S^k \Omega^1_{\mani, odd}$ a differential $k$-superform, or a $k$-form for short. The differential $d: \stsheaf \rightarrow \Omega^1_{\mani, odd}$ lifts to the exterior derivative 
$d : \Omega^{k}_{\mani, odd} \rightarrow \Omega^{k+1}_{\mani, odd}$, which is an odd (nilpotent) superderivation of $\Omega^\bullet_{\mani, odd}$, obeying the Leibniz rule in the form
\bear \label{Leibform}
d (\omega \eta) = d \omega \,\eta + (-1)^{|\omega|} \omega \, d\eta,
\eear
for $\omega \in \Omega^{k}_{\mani, odd}$ and $\eta \in \Omega^{\bullet}_{\mani, odd}$ and where $|\omega|$ is the parity of $\omega$ (which equals the degree of $\omega$ mod 
$\mathbb{Z}_2$: notice that $(-1)^{|\omega|} = (-1)^{\deg (\omega)}$). Here and in what follows we leave the product in the superalgebra of forms understood for the sake of notation. The pair $(\Omega^{\bullet}_{\mani, odd}, d)$ defines the \emph{de Rham complex} of $\mani$. Once again, we will consider any $\Omega^k_{\mani, odd}$ with the structure of 
{right} $\mathcal{O}_\mani$-module. \\
Likewise, we call a section of $ (\Omega^k_{\mani, odd})^\ast \defeq S^k \Pi \mathcal{T}_\mani$ a $\Pi$-vector $k$-field, or a polyvector field for short. Once again, any $(\Omega^k_{\mani, odd})^\ast$ has the structure of {left} 
$\stsheaf$-module. \\
Notice that there exists a pairing 
\bear \xymatrix@R=1.5pt{
\langle \; , \; \rangle : \Omega^{\bullet}_{\mani, odd} \otimes_{\mathcal{O}_\mani} (\Omega^\bullet_{\mani, odd})^\ast  \ar[r] & (\Omega^\bullet_{\mani, odd})^\ast  \\
\omega \otimes \tau \ar@{|->}[r] & \langle \omega, \tau \rangle
}\eear 
which is defined via the {contractions} in a way such that
\bear
\langle dx_a , \pi \partial_{x_\beta} \rangle = (-1)^{(|x_a| + 1)(|x_b| + 1)} \delta_{ab}. 
\eear
In particular, it can be observed that $1$-forms $\omega \in \Omega^1_{\mani, odd}$ act as {superderivations} of $(\Omega^\bullet_{\mani, odd})^\ast$, \emph{i.e.}\ they satisfy the Leibniz rule in the above form. Moreover, explicitly, for 
$ \pi X \in \Pi \mathcal{T}_\mani$ and $\omega = df \in \Omega^1_{\mani, odd}$ one easily finds that
\begin{align}
\langle df , \pi X \rangle = (-1)^{(|X| + 1)(|f| + 1)} X (f).
\end{align} 
where $f \in \mathcal{O}_\mani$ and $X \in \mathcal{T}_\mani$.

\

\noindent Also, we introduce the sheaf $\mathcal{D}_\mani$ of (linear) \emph{differential operators} on $\mani$, which can be abstractly defined as the subalgebra of $\mathcal{E}nd_k (\mathcal{O}_\mani)$ generated by 
$\mathcal{O}_\mani$ and $\mathcal{T}_\mani$. This means that over an open set $U$ one has that the set $\{ x_a , \partial_{x_b} \}$, where $x_a \in \mathcal{O}_\mani \lfloor_U$ and $\partial_{x_a} \in \mathcal{T}_\mani \lfloor_U$ for $a$ ranging over 
both even and odd coordinates, gives a local trivialization of $\mathcal{D}_\mani$ over $U$ and where the following defining relations are satisfied
\bear \label{commweyl}
[x_a, x_b] = 0, \qquad [\partial_{x_a} , \partial_{x_b} ] = 0, \qquad [\partial_{x_a}, x_b] = \delta_{ab},
\eear
where $[\, , \, ]$ is the supercommutator. It follows from equations \eqref{commweyl} that $\mathcal{D}_\mani \lfloor_U$ is isomorphic to the \emph{Weyl superalgebra} of $\mathbb{K}^{p|q}$: the sheaf $\mathcal{D}_\mani$ is thus {noncommutative} rather than just 
supercommutative, something which will play a major role in what follows. It is also worth stressing that $\mathcal{D}_\mani$ admits a {filtration} by the degree of the differential operators such that 
$\mathcal{D}_{\mani}^{(\leq i )} \subseteq \mathcal{D}_{\mani}^{(\leq i+1)}$ for any $i \geq 0 $ and $\mathcal{D}_\mani^{(\leq i )} \cdot \mathcal{D}_\mani^{(\leq j)} \subseteq \mathcal{D}_{\mani}^{(\leq i+j)}$. It is not hard to see that, defining 
$\mbox{gr}^{k} (\mathcal{D}_\mani ) \defeq \mathcal{D}^{(\leq k)}/ \mathcal{D}_{\mani}^{(\leq k-1)}$, one has $\mbox{gr}^{k}(\mathcal{D}_\mani) \cong S^k \mathcal{T}_\mani,$ so that in turn one has 
$\mbox{gr}^{\bullet} \mathcal{D}_\mani \cong S^\bullet \mathcal{T}_\mani$, which can be looked at as a sort of supercommutative approximation of $\mathcal{D}_\mani.$ Finally, notice that $\mathcal{D}_\mani$ is endowed with the structure of 
{$\mathcal{O}_\mani$-bimodule}, \emph{i.e.}\ $\mathcal{D}_\mani$ is a left and right $\mathcal{O}_\mani$-module with the operations given respectively by multiplication to the left and to the right by elements $f \in \stsheaf.$ 

\

\noindent A peculiar construction to supergeometry is the one of \emph{Berezinian sheaf} of a supermanifold. This substitutes the notion of canonical sheaf on an ordinary manifold, which makes no sense on a supermanifold since the 
de Rham complex is not bounded from above. Notice that this sheaf does {not} belong to the de Rham complex, \emph{i.e.}\ it is not made out of ordinary differential forms in $\Omega^1_{\mani, odd}.$ On the other hand, just like the canonical sheaf 
in a purely commutative setting, the Berezinian sheaf can be defined via the Koszul complex, or better its supersymmetric generalization - see \cite{Manin} \cite{Severa} and the recent dedicated paper \cite{NojaRe}; for a different very nice construction in the smooth category see \cite{Ruiperez}. More precisely, given a locally-free 
sheaf $\mathcal{E}$ of rank $p|q$ over a supermanifold 
$\mani$, one defines the Berezinian sheaf $\mathcal{B}er (\mathcal{E})$ of $\mathcal{E}$ to be the locally-free sheaf of rank $\delta_{0, (p+q)\mbox{\scriptsize{mod}} 2} | \delta_{1, (p+q)\mbox{\scriptsize{mod}}2} $ given by 
$\mathcal{B}er (\mathcal{E}) \defeq \mathcal{E}xt^{p}_{S^\bullet \mathcal{E}^\ast} (\stsheaf, S^\bullet \mathcal{E}^\ast)$. In particular, one defines the Berezinian sheaf of the supermanifold $\mani$ to be 
$\mathcal{B}er (\mani) \defeq \mathcal{B}er(\Omega^1_{\mani, odd})^\ast $, \emph{i.e.}\ one has
\bear \label{BerezinianNoja}
\mathcal{B}er (\mani) \defeq \mathcal{H}om_{\mathcal{O}_\mani} (\mathcal{E}xt^q_{S^\bullet \Pi \mathcal{T}_\mani} (\mathcal{O}_\mani, S^\bullet \Pi \mathcal{T}_\mani), \mathcal{O}_\mani) \cong_{loc} \Pi^{q+p}\mathcal{O}_{\mani}. 
\eear
In the following, we will use extensively that if $x=z_1, \ldots, z_p | \theta_1, \ldots, \theta_q$ is a system of local coordinates for $\mani$, then the Berezinian sheaf is locally-generated by the class 
\bear
\varphi(x)=[dz_1 \ldots dz_p \otimes \partial_{\theta_1} \ldots \partial_{\theta_q}]
\eear in the above homology \eqref{BerezinianNoja}, see \cite{NojaRe}. Further, we stress that the Berezinian sheaf will be looked at as a sheaf or 
{right} $\mathcal{O}_\mani$-modules.

\

\noindent The Berezinian sheaf of $\mani$ enters the construction of the so-called \emph{integral forms}, see for example \cite{Manin} or \cite{Penkov}. Given a supermanifold $\mani$, these are defined as sections of the sheaf 
$\mathcal{H}om_{\stsheaf} (\Omega^{\bullet}_{\mani, odd}, \mathcal{B}er (\mani))$, or analogously $\mathcal{B}er (\mani) \otimes_{\stsheaf} S^\bullet \Pi \mathcal{T}_\mani.$ Integral forms can be endowed with the structure of a complex by 
providing a differential $\delta : \mathcal{B}er (\mani) \otimes_{\stsheaf} S^k \Pi \mathcal{T}_\mani \rightarrow \mathcal{B}er (\mani) \otimes_{\stsheaf} S^{k-1} \Pi \mathcal{T}_\mani$, whose definition is quite tricky: well-definedness and 
invariance are indeed far from obvious (see \cite{Manin}, where the differential on integral forms is induced using the notion of \emph{right connection} on $\mathcal{B}er (\mani)$). For this reason the construction of the differential which makes integral forms into an actual 
complex will be discussed further later on in the paper. Here we limit ourselves to say that, 
locally, moving functions to the left of the tensor product $\mathcal{B}er (\mani) \otimes_{\stsheaf} S^k \Pi \mathcal{T}_\mani$, the differential gets written as 
\bear \label{deltaint}
\delta (\varphi(x) f \otimes \pi \partial^I ) = - \sum_a (-1)^{|x_a||f| + |\pi \partial^I |} \varphi (x) (\partial_a f) \otimes \partial_{\pi \partial_a } (\pi \partial^I)  
\eear
where $\varphi (x) $ is the local generating section of $\mathcal{B}er (\mani)$ introduced above, $\pi \partial^I$ is a homogeneous section of $S^k \Pi \mathcal{T}_\mani$, $f$ is a function and where the derivative with respect to the coordinate field $\pi \partial_a$ is 
nothing but 
the contraction of the polyfield with the form dual to $\pi \partial_a$, that is $\langle dx_a , \pi \partial^I \rangle = \partial_{\pi \partial_a} ( \pi \partial^I).$ We will see that this definition is related with the structure of right $\mathcal{D}_\mani$-module of 
$\mathcal{B}er (\mani)$ - first discovered by Penkov in \cite{Penkov} - and, in turn, with its Lie derivative. Finally, we stress that given a $p|q$ dimensional supermanifold $\mani$, it is useful to shift the degree of the complex of integral forms, posing 
$ \Sigma_{\mani}^{p - \bullet} \defeq \mathcal{B}er (\mani) \otimes_{\stsheaf} S^{p-\bullet} \Pi \mathcal{T}_\mani $ and consider the complex $(\Sigma_{\mani}^{p - \bullet }, \delta )$, so that an integral form of degree $p$, \emph{i.e.}\ a section of the Berezinian sheaf, 
can be integrated on $\mani$, in the same fashion as an ordinary $p$-form can be integrated on an ordinary $p$-dimensional manifold. More in general, with this convention, it can be seen that an integral form of degree $p-k$ on $\mani$ can be integrated on a sub-supermanifold of codimension $k|0$ in $\mani$, see for example \cite{Witten}.

%%%%%%%%%%%%%%%%%%%%%%
\section{Universal de Rham Complex and its Homology}

\noindent We now introduce one of the main characters of our study, \emph{cfr.}\ for example \cite{Sch95}.
\begin{definition}[Universal de Rham Sheaf of $\mani$] Let $\mani$ be a supermanifold. We call the sheaf $ \Omega^\bullet_{\mani, odd} \otimes_{\stsheaf} \mathcal{D}_\mani$ the universal de Rham sheaf of $\mani$. 
\end{definition}
\noindent Notice that the universal de Rham sheaf is $\mathbb{Z}$-graded by the gradation of $\Omega^\bullet_{\mani, odd} $ and also $\mathbb{Z}_2$-graded as both of its components are. Moreover it is filtered by the filtration by degree on 
$\mathcal{D}_\mani$ introduced in the previous section.   
\noindent Clearly, the universal de Rham sheaf $\Omega^{\bullet}_{\mani, odd} \otimes_{\stsheaf} \mathcal{D}_\mani$ is naturally a \emph{left} $\Omega^\bullet_{\mani, odd}$-module and a {right} $\mathcal{D}_\mani$-module. In particular it is 
a {left} $\mathcal{O}_\mani$-module by restriction on the structure of $\Omega^{\bullet}_{\mani, odd}$-module. On the other hand $\Omega^\bullet_{\mani, odd} \otimes \mathcal{D}_\mani$ is also a {right} $\mathcal{O}_\mani$-module with 
the structure induced by the one of right $\mathcal{D}_\mani$-module: this structure, though, does {not} coincide with the one of left $\mathcal{O}_\mani$-module.

\

\noindent We are interested in finding a \emph{natural} differential as to make the universal de Rham sheaf into a proper complex of sheaves. We first need the following 
\begin{definition}[$\mathcal{O}_\mani$-Definition] \label{defOM} Let $\mathcal{L}$ and $\mathcal{R}$ be a left and a right $\stsheaf$-module respectively. Let $\phi : \mathcal{R} \otimes_{\mathbb{C}} \mathcal{L} \rightarrow \mathcal{H}$ be a morphism of sheaves of 
$\mathbb{C}$-modules into a sheaf $\mathcal{H}$. We say that $\phi$ is $\mathcal{O}_\mani$-defined if it descends to a $\mathbb{C}$-linear operator $\hat \phi : \mathcal{R} \otimes_{\mathcal{O}_\mani} \mathcal{L} \rightarrow \mathcal{H}$, \emph{i.e.}\ if 
the identity 
\bear
\phi (l f \otimes r) = \phi (l \otimes f r)
\eear 
holds true for any $l \in \mathcal{L}, r \in \mathcal{R}$ and $f \in \stsheaf.$
\end{definition}
\noindent Given this definition, we now introduce the following operator
\begin{definition}[The Operator $D$] Let $\omega \otimes F \in \Omega^\bullet_{\mani, odd} \otimes_{\stsheaf} \mathcal{D}_\mani$ such that $\omega $ and $F$ are homogeneous. We let $D$ be the operator 
\bear
\xymatrix@R=1.5pt{
D : \Omega^\bullet_{\mani, odd} \otimes_{\mathbb{C}} \mathcal{D}_\mani \ar[r] & \Omega^\bullet_{\mani, odd} \otimes_{\stsheaf} \mathcal{D}_\mani \\
\omega \otimes F \ar@{|->}[r] &  D \left (\omega \otimes F \right ) \defeq d \omega \otimes F + \sum_{a} (-1)^{|\omega| |x_a|} dx_a  \omega \otimes \partial_{x_a}  \cdot F,
 }
\eear
where $x_a = z_1, \ldots, z_p | \theta_1, \ldots, \theta_q$, so that the index $a$ runs over all of the even and odd coordinates. 
\end{definition}
\noindent Notice that the operator $D$ is of degree $+1$ with respect to the $\mathbb{Z}$-degree of $\Omega^{\bullet}_{\mani, odd}$, \emph{i.e.}\ it raises the form number by one. The properties of $D$ are characterized in the following Lemma. 
\begin{lemma} \label{propD} The operator $D$ has the following properties:
\begin{enumerate}
\item it is globally well-defined \emph{i.e.}\ it is invariant under generic change of coordinates; 
\item it is $\mathcal{O}_\mani$-defined in the sense of definition \ref{defOM}, \emph{i.e.}\ it induces an operator 
$D : \Omega^\bullet_{\mani, odd} \otimes_{\stsheaf} \mathcal{D}_\mani \otimes_{\stsheaf} (\Omega^\bullet_{\mani})^{\ast} \rightarrow \Omega^{\bullet}_{\mani, odd} \otimes_{\stsheaf} \mathcal{D}_\mani \otimes_{\stsheaf} (\Omega^\bullet_{\mani})^{\ast};$
\item it is nilpotent, \emph{i.e.}\ $D^2 = 0.$
\end{enumerate}
\end{lemma}
\begin{proof}
We prove separately the three claims of the Lemma.
\begin{enumerate}
\item Obvious, since each of the two summands is invariant by itself.     
\item We prove that for any $f \in \mathcal{O}_\mani,\,  \omega \in \Omega^\bullet_{\mani, odd}, \, F \in \mathcal{D}_\mani$ we have $D (\omega f \otimes_{\mathbb{C}} F ) = D (\omega \otimes_{\mathbb{C}} f F).$
Indeed, posing $df = \sum_a dx_a \partial_{x_a} f$, on the one hand one computes
\begin{align} \label{Od1}
D (\omega f \otimes F ) & = (d\omega) f \otimes F + (-1)^{|\omega|} \omega (df) \otimes F + \sum_{a} (-1)^{|x_a|(|\omega| + |f|)}dx_a \omega f \otimes \partial_{x_a} F  \nonumber \\
 & = (d\omega) f \otimes F + (-1)^{|\omega|} \omega \sum_a dx_{\alpha} \partial_{x_a} f \otimes F + \sum_{a} (-1)^{|x_a|(|\omega| + |f|)}dx_a \omega f \otimes \partial_{x_a} F. 
\end{align}
On the other hand, one has
\begin{align} \label{Od2}
D (\omega \otimes f F) & = d\omega \otimes f F +  \sum_{a} (-1)^{|x_a| |\omega| } dx_a \omega \otimes  \left ( (\partial_{x_a } f) F + (-1)^{|x_a| |f| } f (\partial_{x_a} F) \right ) \nonumber \\
& = (d\omega) f \otimes F + (-1)^{|\omega|} \omega \sum_a dx_{a} \partial_{x_a} f \otimes F + \sum_{a} (-1)^{|x_a|(|\omega| + |f|)}dx_a \omega f \otimes \partial_{x_a} F,
\end{align}
so that \eqref{Od1} is matched by \eqref{Od2}.
\item We prove that $D^2 = 0$. Writing $D = D_1 + D_2$, with $D_1 \defeq d \otimes 1 $ and $D_2 \defeq \sum_a dx_a \otimes \partial_{x_a}$ one has that 
$D^2 = D^2_1 + (D_1 D_2 + D_2 D_1)+ D^2_2.$ Clearly, $D_1^2 = 0 $ and $D_2^2 = 0 $ as well, for it is an {odd} element in the supercommutative algebra $\mathbb{C}[dx_a] \otimes_{\mathbb{C}} \mathbb{C}[\partial_a]$. It remains to prove that 
$[ D_1, D_2 ] \defeq D_1 D_2 + D_2 D_1=0$. \\
We have
\begin{align}
D_2 D_1 (\omega \otimes F) = \sum_{a} (-1)^{|x_a| ( |\omega| + 1)  } dx_{a} d\omega \otimes \partial_{x_a} F. 
\end{align}
One the other hand, one finds
\begin{align}
D_1 D_2 (\omega \otimes F) 
& = \sum_a (-1)^{|x_a|(|\omega| + 1) +1 } dx_a d\omega \otimes \partial_{x_a} F 
\end{align}
which cancels exactly the previous expression for $D_2 D_1.$
\end{enumerate}
\end{proof}
\noindent The above Lemma justifies the following definition. 
\begin{definition}[Universal de Rham Complex of $\mani$] Let $\mani$ be a supermanifold. We call the pair $(\Omega_{\mani, odd}^\bullet \otimes_{\stsheaf} \mathcal{D}_\mani, D)$ the universal de Rham complex of $\mani.$
\end{definition}
\noindent We now compute the homology of this complex.
\begin{theorem}[Homology of the Universal de Rham Complex]\label{UniversalDR} Let $\mani$ be a supermanifold and let $(\Omega^{\bullet}_{\mani, odd} \otimes_{\stsheaf} \mathcal{D}_\mani, D)$ be the universal de Rham complex of $\mani$. There 
exists a canonical isomorphism of sheaves
\bear
H_\bullet (\Omega^\bullet_{\mani, odd} \otimes_{\stsheaf} \mathcal{D}_\mani, D ) \cong \mathcal{B}er (\mani),
\eear
where $\mathcal{B}er (\mani)$ is the Berezinian sheaf of $\mani.$
\end{theorem}
\begin{proof} The proof can be done by constructing a homotopy for the operator $D$. Clearly, the first part of $D$, namely $D_1 = d \otimes 1$ has the usual homotopy of the de Rham complex. By the way elements of the form $c \otimes F$, for $c$ a 
constant and $F $ a generic element in $\mathcal{D}_\mani$ are not in the kernel of $D$.\\
Let us now look at the second summand, $D_2 (\omega \otimes F) = \sum_a (-1)^{|\omega| |x_a|} dx_a \omega \otimes \partial_a F$. We work in a chart $(U, x_a)$ such that the sheaf $\Omega^\bullet_{\mani, odd} \otimes_{\stsheaf} \mathcal{D}_\mani$ can 
be represented as the sheaf of vector spaces generated by the monomials of the form $ \omega \otimes F $ with $\omega = dx^I$, $F = \partial^J f$ for some multi-indices $I$ and $J$ and some $f \in \stsheaf \lfloor_{U}$. We define the following operator 
on $ \Omega^\bullet_{\mani, odd} \otimes \mathcal{D}_\mani \lfloor_{U}$ 
\bear
H (\omega \otimes F) \defeq \sum_a (-1)^{|x_a|( |\omega| + |\partial^J| + 1)}{\partial_{dx_a}} dx^I \otimes [\partial^J , x_a] f,
\eear
where the derivation $\partial_{dx_a}$ can be seen as the contraction with respect to the coordinate vector field $\partial_a$ (up to a sign). We claim that $H$ is a homotopy. By explicit computation, one has
\begin{align}
D_2 H (\omega \otimes F) 
& = \sum_{a,b} (-1)^{|x_a| (|x_b| + 1 + |\omega|) + |x_b|(|\omega| + |\partial^J| + 1)} dx_a \partial_{dx_b} \omega \otimes  \partial_{a} [\partial^J, x_b] f, 
\end{align}
\begin{align} \label{hd2}
H D_2 (\omega \otimes F) 
& = \sum_{a, b} (-1)^{|x_b| (|\omega| + |\partial^J|) + |x_a| |\omega|} \partial_{dx_b} (dx_a \omega) \otimes [\partial_b \partial^J , x_a ] f.
\end{align}
Expanding the above equation \eqref{hd2}, one finds
\begin{align} \label{hd3}
H D_2 (\omega \otimes F) & = - D_2 H (\omega \otimes F) + \sum_{a,b} (-1)^{|\omega| (|x_a| + |x_b|)}\delta_{a b} \omega \otimes \partial^J f + \nonumber \\
& \quad +  \sum_a (-1)^{|x_a||\partial^J|} \omega \otimes \partial_a [\partial^J, x_a] f + \sum_a (-1)^{|x_a|+ 1} dx_a (\partial_{dx_a} \omega) \otimes \partial^J f .
\end{align}
We now analyze the summands in \eqref{hd3}. If $x_a = z_1, \ldots, z_p | \theta_1, \ldots, \theta_q $, recalling that $\omega = dx^I$, we define $\mbox{deg}_0 (\omega)$ to be the degree of $\omega$ with respect to the even generators 
($d\theta$'s) and $\mbox{deg}_1 (\omega)$ to be the degree of $\omega$ with respect to the odd generators ($dz$'s) and likewise we pose $\mbox{deg}_0 (\partial^J) $ to be the degree of $\partial^J$ with respect to the even generators ($\partial_z$'s) 
and $\mbox{deg}_1 (\partial^J) $ to be the degree of $\partial^J$ with respect to the odd generators ($\partial_\theta$'s). With these definitions, one can observe that 
\begin{align}
& \sum_{ab} (-1)^{|\omega|(|x_a| + |x_b|)} \delta_{ab} \omega \otimes \partial^J f = (p+q) (\omega \otimes F),  \nonumber \\
& \sum_{a} (-1) ^{|x_a||\partial^J|} \omega \otimes \partial_a [\partial^J, x_a] f = \sum_a (-1)^{|x_a|} \omega \otimes \partial^J f = ( \mbox{deg}_0 (\partial^J) -  \mbox{deg}_1 (\partial^J) )(\omega \otimes F), \nonumber \\
& \sum_{a} (-1)^{|x_a|+ 1} dx_a (\partial_{dx_a} \omega) \otimes \partial^J f = ( \mbox{deg}_0 (\omega) -  \mbox{deg}_1 (\omega) )(\omega \otimes F).
\end{align}
This yields
\begin{align}
(HD_2 + D_2 H) (\omega \otimes F) = \left ( p + q + \mbox{deg}_0 (\omega) + \mbox{deg}_0 (\partial^J) - \mbox{deg}_1 (\omega) - \mbox{deg}_1 (\partial^J) \right ) (\omega \otimes F).
\end{align}
It follows that $(HD_2 + D_2 H) (\omega \otimes F) = c \cdot (\omega \otimes F)$ for $c$ a constant, which proves the claim that $H$ defines a homotopy if $c \neq 0$.
The homotopy fails in the case $c= 0$. In particular, note that, by anticommutativity, $\mbox{deg}_1 (\omega) \leq p $ and $\mbox{deg}_1 (\partial^J) \leq q$, since there can only be $p$ odd forms $dz_1\cdot  \ldots \cdot dz_p$ and $q$ odd derivation 
$\partial_{\theta_1} \cdot \ldots \cdot \partial_{\theta_q}$, hence $c=0$ if and only if
\bear
\left \{ 
\begin{array}{l}
\mbox{deg}_0 (\omega) = \mbox{deg}_0 (\partial^J) = 0\\
\mbox{deg}_1 (\omega) = p \\
\mbox{deg}_1 (\partial^J) = q.
\end{array}
\right.
\eear
In this case the monomial $ \omega \otimes F$ is of the form 
$ dz_1 \ldots dz_p \otimes \partial_{\theta_1} \ldots \partial_{\theta_q} \cdot f$
for $f \in \mathcal{O}_\mani \lfloor_{U}$: this element generates the Berezinian sheaf $\mathcal{B} er(\mani)$ and it is non-zero in the homology $H_\bullet (\Omega^\bullet_{\mani, odd} \otimes \mathcal{D}_\mani, D)$, thus concluding the proof.\end{proof}
{\remark We observe that the previous Theorem holds true in any \virgolette geometric'' category: $\mani$ might be a {real smooth} or a {complex analytic} supermanifold, but also an {algebraic} supermanifold.}
{\remark It is proved in \cite{Penkov} that the Berezinian sheaf of a supermanifold carries a structure of {right} $\mathcal{D}_\mani$-module. This is constructed via the action of the Lie derivative on sections of $\mathcal{B}er (\mani),$ which 
somehow parallels the analogous result on the canonical sheaf ${K}_{M}$ of an ordinary manifold $M$. Indeed, it is an easy application of Cartan calculus to see that if $\omega$ is a section of the canonical sheaf of $M$, with local trivialization given by 
$\omega(x) f \defeq dx_1 \wedge \ldots dx_p f$, for $f \in \mathcal{O}_\mani$, then $\mathcal{L}_X (\omega ) = \omega (x) \sum_a \partial_{a} (f X^i) $ for any vector field $X = \sum_i X^i \partial_i.$ It is then not difficult to show that defining a right action 
$K_M \otimes {T}_M \rightarrow K_M$ on vector fields as $\omega \otimes X \mapsto \omega \cdot X \defeq - \mathcal{L}_{X} (\omega )$ endows $K_M$ with the structure of {right} $\mathcal{D}_M$-module - indeed the former action defines a 
\emph{flat} right connection on $K_M$, see \cite{Manin}. 
The same holds true in the case of the Berezinian sheaf on a supermanifold, but the construction of the action of the Lie derivative is not as straightforward as in the ordinary case, since the Berezinian is {not} a sheaf of forms and therefore there is no obvious 
generalization of the Cartan calculus on it. Nonetheless, it can be shown - for example via an analytic computation using the flow along a vector field (see also \cite{Penkov}) - that 
\bear \label{LieBer}
\mathcal{L}_X (\varphi) = (-1)^{|\varphi (x)||X|} \varphi (x) \sum_a (-1)^{|X^a| (|x_a| + |f|)}\sum_a \partial_a (fX^a),   
\eear
where $\varphi$ is a section of the Berezinian sheaf with local trivialization given by $\varphi = \varphi (x) f$, where $f \in \stsheaf$ and where $\varphi(x)$ is the generating section of $\mathcal{B}er (\mani)$. Notice that this can be re-written, more 
simply, as $\mathcal{L}_X (\varphi) = (-1)^{|\varphi||X|} \varphi(x) \sum_a (fX^a) \overset{\leftarrow}{\partial_a}$ if one lets the derivative acts from the right, borrowing the notation from physics. The right action of vector fields making $\mathcal{B}er (\mani)$ 
into a sheaf of right $\mathcal{D}_\mani$-modules is then defined as \cite{Penkov} 
\bear
\xymatrix@R=1.5pt{
\mathcal{B}er (\mani) \otimes \mathcal{T}_\mani \ar[r] & \mathcal{B}er (\mani) \\
\varphi \otimes X \ar@{|->}[r] & \varphi \cdot X \defeq - (-1)^{|\varphi||X| }\mathcal{L}_X (\varphi).
}
\eear
It follows that, taking into account the action \eqref{LieBer} of the Lie derivative, one gets:
\bear \label{actionLie}
\varphi \cdot X = - \varphi(x)\sum_a (-1)^{|x_a| (|X^a| + |f|)}  \partial_a (fX^a).
\eear
It is worth noticing that the above construction, which might look somewhat artificial at first sight, comes for free from the homology of the universal de Rham complex as in Theorem \ref{UniversalDR}. The action of the Lie derivative on sections of the Berezianian 
emerges naturally and effortlessly as a consequence of the fact 
that we are working \emph{ab initio} with a complex of $\mathcal{D}_\mani$-modules. 
Indeed, the previous Theorem \ref{UniversalDR} has the following easy Corollary. }
\begin{corollary}[$\mathcal{B}er (\mani)$ is a Right $\mathcal{D}_\mani$-Module / Lie Derivative] Let $\mani$ be a supermanifold. The right action 
\bear
H_\bullet (\Omega^\bullet_{\mani, odd} \otimes_{\stsheaf} \mathcal{D}_\mani, D ) \otimes_{\stsheaf } \mathcal{D}_\mani \longrightarrow H_\bullet (\Omega^\bullet_{\mani, odd} \otimes_{\stsheaf} \mathcal{D}_\mani, D )
\eear
is uniquely characterized by $\varphi(x) \cdot \partial_a \defeq [dz_1 \ldots dz_p \otimes \partial_{\theta_1} \otimes \ldots \partial_{\theta_q}] \cdot \partial_a = 0 $ for any $a$, and it is given by the Lie derivative on $\mathcal{B}er (\mani).$
\end{corollary}
\begin{proof} One easily checks that in the homology of $D$ one has $[dz_1 \ldots dz_p \otimes \partial_{\theta_1} \otimes \ldots \partial_{\theta_q} \partial_a ] = 0 $ for any $a$, which characterizes the right action of $\mathcal{D}_\mani$ on 
$H_\bullet(\Omega^\bullet_{\mani, odd} \otimes_{\stsheaf} \mathcal{D}_\mani, D ) \cong \mathcal{B}er (\mani).$ \\
Explicitly, for a section $\varphi = dz_1 \ldots dz_p \otimes \partial_{\theta_1} \otimes \ldots \partial_{\theta_q}  f $ of the Berezinian, and a generic vector fields $X = \sum_a X^a \partial_a $, one computes using the 
$\mathcal{D}_\mani$-module structure
\begin{align} \label{lieber1}
\varphi \cdot X & = \varphi(x) \sum_a (-1)^{|x_a| (|X^a| + |f|)} \left ( - \partial_a (fX^a) + \partial_a \cdot f X^a \right ) \nonumber \\
& = -\varphi(x) \sum_a (-1)^{|x_a| (|X^a| + |f|)} \partial_a (f X^a),
\end{align} 
where we have used that the second summand is zero in the homology, so that \refeq{lieber1} matches the previous \eqref{actionLie}.
\end{proof}
\noindent Before we pass to the next section, let us stress that \cite{Manin} offers a different but related point of view, closer to the one given in \cite{Penkov}, where the notion of $\mathcal{D}_\mani$-module, and in particular the construction of the 
$\mathcal{D}_\mani$-module structure on $\mathcal{B}er (\mani)$ is left understood, but implied by the exposition. To retrive the $\mathcal{D}_\mani$-module structure from \cite{Manin} one would further need to prove that the right connection defined on 
$\mathcal{B}er (\mani)$ is flat: this actually coincide with 
the \eqref{LieBer}.

\section{Universal Spencer Complex and its Homology}

\noindent We now repeat the above construction by using the dual $(\Omega^\bullet_{\mani, odd})^\ast$ of the de Rham complex instead. We start with the following definition, \emph{cfr}.\ \cite{Sch95}
\begin{definition}[Universal Spencer Sheaf of $\mani$] Given a supermanifold $\mani$, we call the sheaf $ \mathcal{D}_\mani \otimes_{\stsheaf} (\Omega^\bullet_{\mani, odd})^\ast$ the universal Spencer sheaf of $\mani$.
\end{definition}
\noindent Just like above, we would like to make the universal Spencer sheaf into an actual complex, by introducing a nilpotent differential on it and then computing its homology. We will see that this differential is more complicated with respect to the previous operator $D$ for the universal de Rham complex. \\
 In order to get such a differential, we first need to study the Lie derivative on the polyfields $(\Omega^\bullet_{\mani, odd})^\ast$. These can be defined recursively as follows.
\begin{definition}[Lie Derivate on $(\Omega^{\bullet}_{\mani, odd} )^\ast$] \label{LieStar} Let $X \in \mathcal{T}_\mani$ be a vector field. The Lie derivative 
$\mathfrak{L}_X : (\Omega^{\bullet}_{\mani, odd})^\ast \rightarrow (\Omega^{\bullet}_{\mani, odd})^\ast$ are defined recursively via the following relations
\begin{enumerate}
\item $\mathfrak{L}_X (f) = X (f) = \mathcal{L}_X (f)$ for any $ f \in \mathcal{O}_\mani$ where $\mathcal{L}_X$ is the usual Lie derivative;
\item Having already defined $\mathfrak{L}_X : (\Omega^{h}_{\mani, odd})^\ast \rightarrow (\Omega^{h}_{\mani, odd})^\ast$ for $ h<k $, one uniquely defines $\mathfrak{L}_X$ on $(\Omega^{k}_{\mani, odd})^\ast$ via the relation
\bear
\mathfrak{L}_X (\langle \omega, \tau \rangle) = \langle \mathcal{L}_X (\omega) , \tau \rangle + (-1)^{|\omega| |X|} \langle \omega, \mathfrak{L}_X (\tau) \rangle \qquad \forall \omega \in \Omega^{\bullet >0}_{\mani, odd}.   
\eear
\end{enumerate}
\end{definition}
\noindent The following Lemma characterizes the properties of the Lie derivative on $(\Omega^{\bullet}_{\mani, odd})^\ast$. For the sake of the exposition, we have deferred its proof to the Appendix. 
\begin{lemma} \label{PropLie} The Lie derivative $\mathfrak{L}_X : (\Omega^{\bullet}_{\mani, odd})^\ast \rightarrow (\Omega^{\bullet}_{\mani, odd})^\ast$ has the following properties: 
\begin{enumerate}
\item $\mathfrak{L}_X (\tau) = \pi [X , \pi \tau]$ for any $\tau \in \Pi \mathcal{T}_\mani$;
\item $\mathfrak{L}_X$ is a superderivation of $(\Omega^{\bullet}_{\mani, odd})^\ast$, \emph{i.e.}\ the super Leibniz rule holds true:
\bear \label{superd}
\mathfrak{L}_X  (\tau_1 \tau_2) = \mathfrak{L}_X (\tau_1) \tau_2 + (-1)^{|X||\tau_1|} \tau_1 \mathfrak{L}_{X} (\tau_2)
\eear
for any $\tau_1, \tau_2 \in (\Omega^{\bullet}_{\mani, odd})^\ast$ and $X \in \mathcal{T}_\mani;$
\item $\mathfrak{L}_{fX} (\tau) = f \mathfrak{L}_{X} (\tau) + (-1)^{|X||f|} \pi X \langle df, \tau \rangle$ for any $f \in \mathcal{O}_\mani$, $\tau \in (\Omega^{\bullet}_{\mani, odd})^\ast$.
\end{enumerate}
\end{lemma}
\noindent Now, using the Lie derivative on $(\Omega^{\bullet}_{\mani, odd})^\ast$ we introduce the following \emph{local} operator. 
\begin{definition}[The Operator $\mathfrak{e}_x $] Let $\tau \in (\Omega^{\bullet}_{\mani, odd})^\ast$. We let the operator $\mathfrak{e}_x$ be defined as 
\bear \xymatrix@R=1.5pt{
\mathfrak{e}_x : (\Omega^{\bullet}_{\mani, odd})^\ast \ar[r] & (\Omega^{\bullet}_{\mani, odd})^\ast \\
\tau \ar@{|->}[r] & \mathfrak{e}_x (\tau) \defeq \sum_a \langle dx_a , \mathfrak{L}_{\partial_a} (\tau) \rangle. 
}
\eear
\end{definition}
{\remark We observe that $\mathfrak{e}_x$ is {not} invariant under general change of coordinates. Also, it is {not} a derivation. On the other hand, it has the following property
\bear \label{prop_e}
\mathfrak{e}_x (f\tau) = (-1)^{|f|} f \mathfrak{e}_x (\tau) + \sum_a (-1)^{|f| (|x_a| + 1)} (\partial_a f)  \langle dx_a , \tau\rangle,
\eear
which follows from a direct computation.}
\noindent We now introduce the following fundamental operator.
\begin{definition}[The Operator $\delta$] Let $F \otimes \tau \in \mathcal{D}_\mani \otimes_{\stsheaf} (\Omega^{\bullet}_{\mani, odd})^\ast$ such that $F$ and $\tau$ are homogeneous. We let $\delta$ be the operator 
\bear
\xymatrix@R=1.5pt{
\delta :  \mathcal{D}_\mani \otimes_{\mathbb{C}} (\Omega^{\bullet}_{\mani, odd})^\ast  \ar[r] & \mathcal{D}_\mani \otimes_{\stsheaf} (\Omega^{\bullet}_{\mani, odd})^\ast \\
F\otimes \tau \ar@{|->}[r] &  \delta \left (F \otimes \tau \right) \defeq  (-1)^{|\tau|} F \sum_a\partial_{a} \otimes \langle dx_a, \tau \rangle - (-1)^{|\tau|} F \otimes \mathfrak{e}_{x} (\tau),
 }
\eear
where the index $a$ runs over all of the even and odd coordinates. 
\end{definition} 
\noindent Notice that, differently from the operator $D$ on the universal de Rham complex, it is not at all apparent whether the operator $\delta$ is well-defined globally on $\mani$ or it is just a local operator. In the following Lemma, which is the 
analogous 
of Lemma \ref{propD} for $D$, we prove the properties of $\delta$. In particular, we prove that $\delta$ is invariant: this happens because of a \virgolette magical'' cancellation between the two transformed summands that appear in the definition of $\delta$, which are clearly not invariant when taken alone. 
\begin{lemma} \label{propdelta} The operator $ \delta$ has the following properties:
\begin{enumerate}
\item it is globally well-defined, \emph{i.e.}\ it is invariant under generic change of coordinates; 
\item it is $\mathcal{O}_\mani$-defined in the sense of definition \ref{defOM}, \emph{i.e.}\ it induces an operator $\delta : \mathcal{D}_\mani \otimes_{\stsheaf} (\Omega^\bullet_{\mani})^{\ast} \rightarrow \mathcal{D}_\mani \otimes_{\stsheaf} (\Omega^\bullet_{\mani})^{\ast};$
\item it is nilpotent, \emph{i.e.}\ $\delta^2 = 0$.
\end{enumerate}
\end{lemma}
\begin{proof} We prove separately the three claims of the Lemma.  
\begin{enumerate}
\item We start proving the invariance of the operator under general change of coordinates. Adopting Einstein's convention on repeated indices, one computes
\begin{align} \label{eq_1}
\mathfrak{e}_z (\tau) & = \langle dz_b , \mathfrak{L}_{\frac{\partial}{\partial{z_b}}} (\tau) \rangle = \langle dx_a \frac{\partial z_b}{\partial x_a} , \mathfrak{L}_{\frac{\partial x_c}{\partial z_b} \frac{\partial}{\partial x_{c}}} (\tau ) \rangle \nonumber \\
& = \mathfrak{e}_x (\tau) - (-1)^{ |x_a||z_c| + |x_a|} \langle \frac{\partial}{\partial x_a} \left( \frac{\partial x_a}{\partial z_c} \right )  dz_c , \tau \rangle.
\end{align}
On the other hand one has 
$\frac{\partial}{\partial z_b} \otimes \langle dz_b , \tau \rangle  = \frac{\partial x_a}{\partial z_b} \frac{\partial }{\partial x_a}  \otimes \langle dx_c \frac{\partial z_b}{\partial x_c} , \tau \rangle. $
Upon using the $\mathcal{D}_\mani$-module relation
one has that 
$\frac{\partial x_a}{\partial z_b} \frac{\partial}{\partial x_a } = (-1)^{|x_a| + |x_a||z_b| } \left (\frac{\partial}{\partial x_a} \frac{\partial x_a}{\partial z_b}  - \frac{\partial}{\partial x_a} \left ( \frac{\partial x_a}{\partial z_b} \right ) \right )$. Using this, one can 
compute that 
\begin{align} \label{eq_2}
\frac{\partial}{\partial z_b} \otimes \langle dz_b , \tau \rangle
& = \frac{\partial}{\partial x_a} \otimes \langle dx_a , \tau \rangle - (-1)^{|x_a| + |x_a||z_b|} \langle \frac{\partial}{\partial x_a} \left ( \frac{\partial x_a}{\partial z_b}\right ) dz_b, \tau \rangle. 
\end{align}
Putting together equations \eqref{eq_1} and \eqref{eq_2} one finds
\begin{align}
 \delta_z (\omega \otimes F \otimes \tau ) &= (-1)^{|\tau|} \omega \otimes F \partial_{a} \otimes \langle dx_a, \tau \rangle -  (-1)^{|\tau |} (-1)^{|x_a| + |x_a||z_b|} \langle \frac{\partial}{\partial x_a} \left ( \frac{\partial x_a}{\partial z_b}\right ) dz_b, \tau \rangle 
 \nonumber \\
 & - (-1)^{|\tau|}\omega \otimes F \otimes \mathfrak{e}_{x} (\tau) + (-1)^{|\tau|} (-1)^{|x_a| + |x_a||z_b|} \langle \frac{\partial}{\partial x_a} \left ( \frac{\partial x_a}{\partial z_b}\right ) dz_b, \tau \rangle \nonumber\\
 & =  \delta_x (\omega \otimes F \otimes \tau ),
\end{align} 
thus completing the proof of invariance.
\item We now prove the $\mathcal{O}_\mani$-definedness. We once again adopt Einstein convention on repeated indices. On the one hand one has
\begin{align}
\delta ( F f \otimes \tau ) & = (-1)^{|\tau|} \left ( F f \partial_{a} \otimes \langle dx_a , \tau \rangle -  F f \otimes \mathfrak{e}_x (\tau)\right ) \nonumber \\
& = (-1)^{|\tau|} \big ( (-1)^{|f||x_a|} F \partial_a \otimes f  \langle dx_a , \tau \rangle - (-1)^{|f| |x_a|} F \otimes \partial_a f \langle dx_a, \tau \rangle + \nonumber \\
& \quad -  F \otimes f \mathfrak{e}_x (\tau) \big ).
\end{align}
On the other hand, one computes
\begin{align}
\delta ( F \otimes f \tau ) & = (-1)^{|f| + |\tau|} \left (  F \partial_a \otimes \langle dx_a , f \tau \rangle  -  F \otimes \mathfrak{e}_x (f\tau) \right ) \nonumber \\
& = (-1)^{|\tau|} \big ( (-1)^{|f||x_a|}  F \partial_a \otimes f \langle dx_a , \tau \rangle - (-1)^{|f||x_a|}  F \otimes \partial_a f \langle dx_a, \tau \rangle + \nonumber \\
& \quad -  F \otimes f \mathfrak{e}_x (\tau ) \big ), 
\end{align}
where we have used the property \eqref{prop_e} above. %This completes the proof of the $\mathcal{O}_\mani$-definedness.
\item We prove that $\delta^2= 0$. In particular, writing again $\delta = \delta_1 + \delta_2 $, posing $\delta_1(F\otimes \tau) \defeq (-1)^{|\tau|} F \sum_a\partial_{a} \otimes \langle dx_a, \tau \rangle $ 
and $\delta_2(F \otimes \tau) \defeq - (-1)^{|\tau|} F \otimes \mathfrak{e}_{x} (\tau)$, it is easy to see that both $\delta_1^2 = 0$ and $\delta^2_2 = 0.$ A direct computation shows that the 
commutator $[\delta_1, \delta_2] = \delta_1 \delta_2 + \delta_2 \delta_1$ vanishes as well, indeed 
\bear
\delta_1 \delta_2 (F \otimes \tau) = F \sum_a \partial_a \otimes \langle dx_a , \mathfrak{e}_x (\tau) \rangle = - \delta_2 \delta_1 (F \otimes \tau),
\eear
\end{enumerate} 
%These conclude the proof.
\end{proof}
\noindent The previous Lemma justifies the following definition.
\begin{definition}[Universal Spencer Complex of $\mani$] Let $\mani$ be a supermanifold. We call the pair $(\mathcal{D}_\mani \otimes_{\stsheaf} (\Omega_{\mani, odd}^\bullet)^\ast, \delta)$ the universal Spencer complex of $\mani.$
\end{definition}
\noindent We now compute the homology of the universal Spencer complex.
\begin{theorem}[Homology of Universal Spencer Complex] \label{USC} Let $\mani$ be a supermanifold and let $( \mathcal{D}_\mani \otimes_{\stsheaf} (\Omega^{\bullet}_{\mani, odd})^\ast, \delta)$ be the universal Spencer complex of $\mani$. There 
exists a canonical isomorphism of sheaves
\bear
H_\bullet ( \mathcal{D}_\mani \otimes_{\stsheaf} (\Omega^\bullet)^\ast_{\mani, odd}, \delta) \cong \stsheaf.
\eear
\end{theorem}
\begin{proof} We construct a homotopy for $\delta$. In particular, we claim that the homotopy is given by
\bear
K (F \otimes \tau) = (-1)^{|\tau|} \sum_a \omega \otimes [F, x_a] \otimes \pi \partial_{a} \cdot \tau
\eear
for $F \in \mathcal{D}_\mani$ and $\tau \in (\Omega^\bullet_{\mani, odd})^\ast.$ First we show that it is $\mathcal{O}_\mani$-defined, indeed one has
\begin{align}
K ( F \otimes f \tau) & = (-1)^{|\tau| + |f|} \sum_a [F, x_a] \otimes \pi \partial_{a} \cdot f \tau \nonumber \\
& =  (-1)^{|\tau| + |f|}  \sum_a (-1)^{|f||x_a| + |f| + |f||x_a|} [Ff, x_a]  \otimes \pi \partial_{a} \cdot \tau \nonumber \\
& = K ( F f \otimes \tau) \label{odefK}
\end{align}
for any $f \in \mathcal{O}_\mani.$ Now, we observe that in general an element of the form $F f \in \mathcal{D}_\mani$ is not homogeneous and, as such, it does not have a well-defined degree. On the other hand, one has that $F f = \sum_j F_j$ with 
$F_j $ homogeneous, so without loss of generality we restrict to elements having a well-defined degree inside $\mathcal{D}_\mani$. Relying on these considerations, we work locally, putting $\tau = f \partial^I$ for some multi-index 
$I$. Using \eqref{odefK} we can write
\bear
K ( F \otimes f \partial^I) = K ( F f \otimes \partial^I) = \sum_{j} K ( F_{j} \otimes \partial^I) 
\eear
and we consider a single term of the sum above. Again, since also $\delta$ is $\mathcal{O}_\mani$-defined, one easily verifies that
\begin{align}
(K \delta + \delta K) ( F_j \otimes \partial^I) & = \sum_a (-1)^{|x_a| + 1}  F_j \otimes \pi \partial_a \langle dx_a, \partial^I \rangle + \sum_a [F_j , x_a] \partial_a \otimes \partial^{I} \nonumber \\
& = (\deg (F_j) + \deg (\partial^J)) ( \omega \otimes F_j \otimes \partial^J).
\end{align}
The homotopy fails in the case $\deg (F_j) + \deg (\partial^J) = 0$, \emph{i.e.}\ when $F$ and $\tau$ are sections of the structure sheaf $\mathcal{O}_\mani,$ which completes the proof.\end{proof}

\section{The de Rham/Spencer Double Complex of a Supermanifold}

\noindent We now aim at getting the previous two sections on the universal de Rham and Spencer complexes together, so that these fit in a unified framework. We start introducing the following definition, which relates the universal de Rham and Spencer complexes. 

\begin{definition}[Sheaf of Virtual Superforms] \label{virtualforms} Let $\mani$ be a supermanifold, we call $$\Omega^\bullet_{\mani, odd} \otimes_{\mathcal{O}_\mani} \mathcal{D}_\mani \otimes_{\mathcal{O}_\mani} (\Omega^\bullet_{\mani, odd})^\ast$$ 
the sheaf of virtual superforms, or virtual forms for short.
\end{definition}
%\remark Before we go on, we observe that in the triple tensor product defining the sheaf of virtual superforms, the noncommutative sheaf $\mathcal{D}_\mani$ is the one which plays the pivotal role, mediating between the sheaf of differential forms and its dual. 
\noindent It is easy to see that the differential $D$ and $\delta$ of the universal de Rham complex and of the universal Spencer complex can be lifted to the whole sheaf of virtual superforms simply by suitably tensoring them by the identity, in particular we will consider 
$D \otimes 1$ and $1 \otimes \delta$. These nilpotent operators commute with each other, as the following Corollary shows.
\begin{corollary} Let $\mani$ be a supermanifold and let $D\otimes 1$ and $1 \otimes \delta$ act on the sheaf of virtual superforms as defined in \ref{virtualforms}. Then $D\otimes 1$ and $1 \otimes \delta$ commute with each other, \emph{i.e.} 
\bear
[1 \otimes \delta, D \otimes 1 ] \defeq (1\otimes \delta) \circ (D \otimes 1) - (D \otimes 1)\circ (1 \otimes \delta ) = 0.
\eear
\end{corollary} \label{corcomm}
\begin{proof} Let $\omega \otimes F \otimes \tau$ be a virtual superform. Then one easily verifies that
\begin{align}
(1 \otimes \delta) \circ (D \otimes 1) (\omega \otimes F \otimes \tau) & = (-1)^{|\tau|} \big ( d\omega \otimes F \partial_a \otimes \langle dx_a , \tau \rangle - d\omega \otimes F \otimes \mathfrak{e} (\tau) + \nonumber \\
& \quad + (-1)^{|x_a||\omega| } dx_a \omega \otimes \partial_a F \partial_b \otimes \langle  dx_b, \tau \rangle  - dx_a \omega \otimes \partial_a F \otimes \mathfrak{e} (\tau) \big ) \nonumber  \\
& = (D\otimes 1) \circ (1 \otimes \delta )(\omega \otimes F \otimes \tau)
\end{align} 
where we have adopted Einstein convention on repeated indices.  
\end{proof}
\noindent Now, for the sake of readability and convenience, we redefine these differentials as to get the following.
\begin{definition}[The Operators $\hat d$ and $\hat \delta$] Let $\omega \otimes F \otimes \tau \in \Omega^\bullet_{\mani, odd} \otimes_{\stsheaf} \mathcal{D}_\mani \otimes_{\stsheaf} (\Omega^\bullet_{\mani, odd})^\ast$ be a virtual superform. We define 
the operators $\hat d$ and $\hat \delta$ as 
\bear
\xymatrix@R=1.5pt{
\hat d : \Omega^\bullet_{\mani, odd} \otimes_{\stsheaf} \mathcal{D}_\mani \otimes_{\stsheaf} (\Omega^\bullet_{\mani, odd})^\ast \ar[r] & \Omega^\bullet_{\mani, odd} \otimes_{\stsheaf} \mathcal{D}_\mani \otimes_{\stsheaf} (\Omega^\bullet_{\mani, odd})^\ast  \\
\omega \otimes F \otimes \tau \ar@{|->}[r] & \hat d (\omega \otimes F \otimes \tau) \defeq (D \otimes 1 ) (\omega \otimes F \otimes \tau),
}
\eear
\bear
\xymatrix@R=1.5pt{
\hat \delta : \Omega^\bullet_{\mani, odd} \otimes_{\stsheaf} \mathcal{D}_\mani \otimes_{\stsheaf} \Omega^\bullet_{\mani, odd} \ar[r] & \Omega^\bullet_{\mani, odd} \otimes_{\stsheaf} \mathcal{D}_\mani \otimes_{\stsheaf} (\Omega^\bullet_{\mani, odd})^\ast  \\
\omega \otimes F \otimes \tau \ar@{|->}[r] & \hat \delta (\omega \otimes F \otimes \tau) \defeq (-1)^{|\omega| +| F| + |\tau|}(1 \otimes \delta ) (\omega \otimes F \otimes \tau).
}
\eear
\end{definition}
\noindent With these definitions one has the following obvious Theorem.
\begin{theorem} \label{doublecomplex} The triple $(\Omega^\bullet_{\mani, odd} \otimes_{\stsheaf} \mathcal{D}_\mani \otimes_{\stsheaf} (\Omega^\bullet_{\mani, odd})^\ast, \hat d, \hat \delta)$ defines a double complex with total differential given by the sum 
$\mathfrak{D} \defeq \hat d + \hat \delta.$
\end{theorem}
\begin{proof}
It is enough to observe that for a section $\eta \in \Omega^\bullet_{\mani, odd} \otimes_{\stsheaf} \mathcal{D}_\mani \otimes_{\stsheaf} (\Omega^\bullet_{\mani, odd})^\ast $ one has that 
\begin{align}
\mathfrak{D}^2 (\eta ) & = ( \hat d^2 + \hat \delta^2 + \hat d \hat \delta + \hat \delta \hat d )(\eta) \nonumber \\
& = (-1)^{|\eta|} \big( (D \otimes 1)\left ( 1 \otimes \delta  \right ) - (1 \otimes \delta) \left ( D \otimes 1\right ) \big )(\eta) \nonumber \\
& = (-1)^{|\eta|} [D\otimes 1,  1 \otimes  \delta ] (\eta) = 0,
\end{align}
thanks to the Corollary \ref{corcomm} and to the fact that $D$ and $\delta$ are nilpotent.
\end{proof}
\noindent The previous Theorem allows us to give the following Definition
\begin{definition}[de Rham/Spencer Double Complex of a Supermanifold] Let $\mani$ be a supermanifold. We call the double complex of sheaves 
$_\mathpzc{D}\mathcal{V}_{\mani}^{\bullet \bullet} \defeq (\Omega^\bullet_{\mani, odd} \otimes_{\stsheaf} \mathcal{D}_\mani \otimes_{\stsheaf} (\Omega^\bullet_{\mani, odd})^\ast, \hat d, \hat \delta)$ the {de Rham/Spencer double complex} or also the virtual superforms double complex.
We define the bi-degrees of the double complex so that the differential $\hat d$ moves {vertically} and $\hat \delta$ moves {horizontally}.  
\end{definition}
{\remark This can be visualized as a {second quadrant} double complex, as $\hat \delta$ lowers the degree in $(\Omega^\bullet_{\mani, odd})^\ast$ by one.}
{\remark Actually, we not only have a double complex but a \emph{triple complex} instead, by taking the \v{C}ech cochains of the above double complex of virtual superforms. Obviously, given any sheaf $\mathcal{F}$ and any open cover 
$\mathcal{U} = \{U\}_{i\in I}$ of $\mani$, the \v{C}ech differential $\check{\delta} : \check{C}^\bullet (\mathcal{U}, \mathcal{F}) \rightarrow \check{C}^{\bullet +1} (\mathcal{U}, \mathcal{F})$ is independent from $\hat D$ and $\hat \delta$, and therefore 
it commutes with both of them, justifying the following Definition.}
\begin{definition}[\v{C}ech-Virtual Superforms Triple Complex] \label{triplecompl} Let $\mani$ be a supermanifold and $\mathcal{U}$ an open cover of $\mani$. We call the triple complex 
$_\mathpzc{T}\mathcal{V}^{\bullet \bullet \bullet}_\mani \defeq(\check{C}^\bullet (\mathcal{U}, \Omega^\bullet_{\mani, odd} \otimes_{\stsheaf} \mathcal{D}_\mani \otimes_{\stsheaf} (\Omega^\bullet_{\mani, odd})^\ast), \check{\delta}, \hat{d}, \hat{\delta})$ the 
\v{C}ech-virtual superforms triple complex or \v{C}ech-virtual superforms complex for short. 
\end{definition} 
\noindent We now study these double and triple complexes. Clearly, to any double complex are attached two spectral sequences.
\begin{definition}[Spectral Sequences $E_r^\Omega$ and $E_r^{\Sigma}$] Let $_\mathpzc{D}\mathcal{V}^{\bullet \bullet}_\mani$ be the virtual superform double complex of $\mani$. We call  
\begin{enumerate}
\item $(E_r^{\Omega}, d^{\Omega}_r)$ the spectral sequence of the virtual superforms double complex with respect to its vertical filtration, \emph{i.e.}\ by first computing homology with respect to the differential $\hat{\delta}$;
\item $(E^\Sigma_r, d^{\Sigma}_r)$ the spectral sequence of the double complex of virtual superforms with respect to its horizontal filtration, \emph{i.e.}\ by first computing homology with respect to the differential $\hat{d}$.
\end{enumerate}
 \end{definition}
\noindent Now, using ordinary spectral sequences machinery, we extract information from the double and triple complex. In particular, we start looking at $E^\Omega_r$: in the next Theorem we show that differential forms arise {at page one}. 
\begin{theorem}[Differential Forms from $_\mathpzc{D}\mathcal{V}_\mani^{\bullet \bullet}$] \label{diffform} Let $\mani$ be a supermanifold and let $(E_r^{\Omega}, d^\Omega_r)$ be the spectral sequence of the double complex 
$_\mathpzc{D}\mathcal{V}^{\bullet \bullet}_\mani$ defined as above. Then \begin{enumerate}
\item
$
E_1^{\Omega} \cong \Omega^{\bullet}_{\mani, odd};$
\item provided that $\mani$ is a real or complex supermanifold, $E_2^{\Omega} = E^\Omega_{\infty} \cong \mathbb{K}_\mani,$
\end{enumerate}
where $\mathbb{K}_{\mani}$ is the constant sheaf valued in the field $ \mathbb{R}$ or $\mathbb{C}$ depending on $\mani$ being real or complex. 
\end{theorem}
\begin{proof} Everything follows easily from previous results. We prove separately the assertions.
\begin{enumerate}
\item By definition $E_1^{\Omega} = H_{\hat \delta} (_\mathpzc{D} \mathcal{V}^{\bullet \bullet}_\mani)$, so it follows from Theorem \ref{USC} that $E_1^{\Omega} \cong \Omega^\bullet_{\mani, odd}$. 
\item Once again, by definition $E_2^\Omega = H_{\hat d} H_{\hat \delta} (_\mathpzc{D} \mathcal{V}^{\bullet \bullet}_\mani)$, \emph{i.e.}\ the cohomology of the de Rham complex $(\Omega^\bullet_\mani, d)$. By the generalization of the Poincar\'e Lemma 
for supermanifolds, see for example \cite{Deligne} or \cite{Manin}, this is isomorphic to $\mathbb{K}_\mathbb{\mani}$. Also, the homology is concentrated in degree zero, so the spectral sequence converges at page two and 
$E_\infty^{\Omega} = E^{\Omega}_2$. 
\end{enumerate}
These conclude the proof.
\end{proof}
\noindent Likewise, we can find integral forms from page one of the other spectral sequence $E^\Sigma_r$, as shown in the following Theorem, mirroring the previous one for $E^\Omega_r$.
\begin{theorem}[Integral Forms from $_\mathpzc{D}\mathcal{V}_\mani^{\bullet \bullet}$] \label{intform} Let $\mani$ be a supermanifold and let $(E_r^{\Sigma}, d^\Sigma_r)$ be the spectral sequence of the double complex 
$_\mathpzc{D}\mathcal{V}^{\bullet \bullet}_\mani$ defined as above. Then \begin{enumerate}
\item
$
E_1^{\Sigma} \cong \mathcal{B}er(\mani)\otimes_{\stsheaf} (\Omega^\bullet_{\mani, odd})^\ast;$
\item provided that $\mani$ is a real or complex supermanifold, $E_2^{\Sigma} = E^\Sigma_{\infty} \cong \mathbb{K}_\mani,$
\end{enumerate}
where $\mathbb{K}_{\mani}$ is the constant sheaf valued in the field $ \mathbb{R}$ or $\mathbb{C}$ depending on $\mani$ being real or complex. 
\end{theorem}
\begin{proof} Just like above, we prove separately the statements.
\begin{enumerate}
\item By definition $E_1^{\Sigma} = H_{\hat d} (_\mathpzc{D} \mathcal{V}^{\bullet \bullet}_\mani)$, so it follows from Theorem \ref{UniversalDR} that $E_1^{\Sigma} \cong \mathcal{B}er (\mani) \otimes_{\stsheaf} (\Omega^\bullet_{\mani, odd})^\ast$. 
\item Again, by definition $E_2^\Sigma = H_{\hat \delta} H_{\hat d} (_\mathpzc{D} \mathcal{V}^{\bullet \bullet}_\mani)$, \emph{i.e.}\ the cohomology of the complex of integral forms $(\Sigma^\bullet_\mani, \delta)$. There is an analogous Poincar\'e Lemma 
for integral forms on supermanifolds (see Theorem 3 in chapter 4, paragraph 8 of \cite{Manin} for the statement, or Theorem \ref{PLInt} next in this section). Once again, this guarantees that, by assigning the degree as explained early on after equation 
\eqref{deltaint}, the homology is isomorphic to $\mathbb{K}_\mathbb{\mani}$ and concentrated in a single degree, so that the spectral sequence converges at page two and $E_\infty^{\Sigma} = E^{\Sigma}_2$. 
\end{enumerate}
These conclude the proof.
\end{proof}
\noindent Finally, using the \v{C}ech-virtual superforms complex, one can prove that differential forms and integral forms compute exactly the same topological invariants related to $\mani$, namely the (co)homology 
$\check{H}^\bullet (\mani, \mathbb{K}_\mani),$ which is actually the cohomology of the total complex. \\
For convenience, for a \emph{real} supermanifold we define
\bear 
H^\bullet_{\mathpzc{dR}} (\mani) \defeq H_d (\Omega^\bullet_{\mani, odd} (\mani)) \quad \mbox{ and }\quad H^\bullet_{\mathpzc{Sp}} (\mani) \defeq H_{\delta} (\Sigma_\mani^\bullet (\mani) ),
\eear 
where $\Omega^{\bullet}_{\mani, odd} (\mani)$ are the global sections of $\Omega^{\bullet}_{\mani,odd}.$ The following Theorem provides the analogous of the 
\emph{\v{C}ech-de Rham isomorphism} in the context of \emph{real} supermanifolds and proves the coincidence of the cohomologies of differential and integral forms (for a categorial construction, see Proposition 1.6.1 and the subsequent remark 
in \cite{Penkov}).  
\begin{theorem}[Equivalence of Cohomology of Differential and Integral Forms] \label{Cechde} Let $\mani$ be a real supermanifold. The cohomology of differential forms $H^\bullet_{\mathpzc{dR}} (\mani)$ and the cohomology of integral forms 
$H^\bullet_{\mathpzc{Sp}} (\mani) $ are isomorphic. In particular, one has
\bear
H^\bullet_{\mathpzc{dR}} (\mani) \cong \check{H}^\bullet (\mani, \mathbb{R}_\mani) \cong H^\bullet_{\mathpzc{Sp}} (\mani).
\eear 
\end{theorem}
\begin{proof} Let us consider the \v{C}ech-de Rham complex $_\mathpzc{T}\mathcal{V}^{\bullet \bullet \bullet}_\mani$. Taking the homology with respect to $\hat \delta$, as done above, one reduces to the usual \v{C}ech-de Rham double complex 
$\check{C}^\bullet (\mathcal{U}, \Omega^\bullet_{\mani, odd}),$ where $\mathcal{U}$ is a good cover of $\mani$. We see that the related spectral sequences converge at page two. Indeed, on one hand 
$H_{\hat d} H_{\check{\delta}} (\check{C}^\bullet (\mathcal{U}, \Omega^\bullet_{\mani, odd})) \cong H_{\mathpzc{dR}}^\bullet (\mani)$, by the generalized Mayer-Vietoris sequence, having used a partition of unity of $\mani$. On the other hand 
$H_{\check \delta} H_{\hat{d}} (C^\bullet (\mathcal{U}, \Omega^\bullet_{\mani, odd})) \cong \check H^\bullet (\mani, \mathbb{R}_\mani),$ by the Poincar\'e lemma. The same argument holds true for 
integral forms with the obvious modifications.\end{proof}

{\remark[Supermanifolds with K\"ahler Reduced Manifold and Hodge-to-de Rham Degeneration] In the above Theorem \ref{Cechde} we have restricted ourselves to the case of \emph{real} supermanifolds. It is quite natural to ask what happens in the case of \emph{complex} supermanifolds. \\
Let us now once again work with differential forms $\Omega^{\bullet}_{\mani, odd}$ - which are now to be seen as {holomorphic} differential forms -, \emph{i.e.}\ computing the cohomology first with respect to $\hat{\delta}$ on the triple complex 
$_\mathpzc{T}\mathcal{V}^{\bullet \bullet \bullet}_\mani.$ Now, the cohomology of the total complex related to the double complex $\check{C}^{\bullet} (\mathcal{U}, \Omega^\bullet_{\mani, odd})$ is just $\check H^\bullet (\mani, \mathbb{C}_\mani)$. Indeed, 
by the holomorphic Poincar\'e Lemma, taking the cohomology with respect to $\hat d$ yields $\check{C}^{\bullet} (\mathcal{U}, \mathbb{C}_\mani)$, so that the spectral sequence converges at page two and the total cohomology is given by 
$\check{H}^\bullet (\mani, \mathbb{C}_\mani).$ \\
On the other hand, there is no holomorphic partition of unity, so that the exactness of the generalized Mayer-Vietoris sequence fails in the complex setting. In this regard, it is a fundamental result in ordinary complex geometry that, for a compact K\"ahler manifold, the 
\emph{Hodge-to-de Rham} (or Fr\"{o}licher) spectral sequence converges at page one, thus giving the decomposition of the de Rham cohomology with complex coefficients $\check{H}(X, \mathbb{C}_\mani)$ into vector spaces of the kind 
$\check{H}^q (X, \Omega^p_X)$, \emph{i.e.}
\bear
\check{H}^n (X, \mathbb{C}_X) \cong \bigoplus_{p+q = n} \check{H}^q (X, \Omega^p_X),
\eear
where $X$ is a generic complex manifold, see for example \cite{Voisin}. Remarkably, in complex supergeometry the Hodge-to-de Rham spectral sequence does {not} converge at page one. Quite the opposite with respect to the commutative case, 
there are many non-zero maps at page one of the spectral sequence. We shall see this by means of an example.}
\begin{example} \label{example}
Let $(\mathpzc{E}, \mathcal{O}_\mathpzc{E})$ be an \emph{elliptic curve} over the complex numbers. We consider a supermanifold ${\mathpzc{SE}} $ of dimension $1|1$ constructed over the elliptic curve $\mathpzc{E}$, whose structure sheaf is given by 
the direct sum of invertible sheaves $\mathcal{O}_{\mathpzc{SE}} \defeq \mathcal{O}_{\mathpzc{E}} \oplus \Pi \Theta_\mathpzc{E}$, where $\Theta_{\mathpzc{E}}$ is a \emph{theta characteristic} of $\mathpzc{E}$, \emph{i.e.} 
$\Theta^{\otimes 2}_{\mathpzc{E}} \cong \mathcal{K}_{\mathpzc{E}}$, for $\mathcal{K}_{\mathpzc{E}}$ the canonical sheaf of $\mathpzc{E}$. We recall that since $\mathpzc{E}$ is an elliptic curve we have 
$\mathcal{K}_\mathpzc{E} \cong \mathcal{O}_{\mathpzc{E}}.$ The supermanifold constructed this way is a genus $g=1$ \emph{super Riemann surface}, also said $\mathcal{N}=1$ \emph{SUSY curve} (of genus 1) \cite{Fior} \cite{WittenRiem}. Over an 
elliptic curve $\mathpzc{E}$ there are four different possible choices for a theta characteristic, three of them are such that $h^0 (\mathpzc{E}, \Theta_{\mathpzc{E}}) = 0$, \emph{i.e.}\ they are {even} theta characteristics, and the remaining one is 
such that $h^0 (\mathpzc{E}, \Theta_{\mathpzc{E}}) = 1$, \emph{i.e.}\ there is a unique {odd} theta characteristic (and it can be identified by $\Theta_{\mathpzc{E}} \cong \mathcal{O}_{\mathpzc{E}}$). Here we have denoted with $h^i$ the dimension 
of the related cohomology group.\\ Let in particular $\mathpzc{SE}$ be the genus $g=1$ super Riemann surface with the choice of the {odd} theta characteristic. We are interested in study the Hodge-to-de Rham spectral sequence related to 
$\check{C}(\mathcal{U}, \Omega^\bullet_{\mathpzc{SE}, odd})$. Computing the cohomology with respect to the \v{C}ech differential, one gets at page one that $E^{p,q}_1 = \check{H}^q ( \Omega^{p}_{\mathpzc{SE}, odd})$. So one is left to study 
the maps $\check{H}^i (\Omega^k_{\mathpzc{SE}, odd}) \rightarrow \check{H}^i (\Omega^{k+1}_{\mathpzc{SE}, odd})$ for $i= 0, 1$ and $k\geq 0$, induced by the de Rham differential on the $\mathbb{C}$-vector spaces 
$\check{H}^i (\Omega^k_{\mathpzc{SE}, odd})$. In order to do this, we note that $\mathpzc{SE}$ is \emph{split} and one has a decomposition of $\Omega^{k}_{\mathpzc{SE}, odd}$ as sheaf of $\mathcal{O}_{\mathpzc{E}}$-modules. It is not hard to see that 
\bear
\Omega^{k\geq 1}_{\mathpzc{SE}, odd} \cong \Theta_{\mathpzc{E}}^{\otimes {k}} \oplus \Theta_{\mathpzc{E}}^{\otimes {k+2}} \oplus \left (\Pi \Theta_{\mathpzc{E}}^{\otimes k+1} \right )^{\oplus 2} \cong \mathcal{O}_{\mathpzc{E}}^{\oplus 2} \oplus 
\Pi \mathcal{O}_{\mathpzc{E}}^{\oplus 2 },
\eear 
where we have used that $\Theta_\mathpzc{E} \cong \mathcal{O}_{\mathpzc{E}} $ and that $\Theta_{\mathpzc{E}} \cong ( \Omega^1_{\mathpzc{SE}, odd} / (\Pi \Theta_{\mathpzc{E}}) \Omega^1_{\mathpzc{SE}, odd})_0$, see \cite{Manin}. This 
decomposition allows one to easily compute the cohomology, which for both $ q= 0 $ and $q = 1$ reads:
\begin{align}
&\check{H}^q (\mathcal{O}_{\mathpzc{SE}}) \cong \mathbb{C}^{1|1}, \qquad \check{H}^q (\Omega^{k\geq 1}_{\mathpzc{SE}, odd}) \cong \mathbb{C}^{2|2}.
\end{align}
Looking at the Hodge-to-de Rham spectral sequence, one is led to study the cohomology of the following sequence of maps of $\mathbb{C}$-vector spaces induced by the de Rham differential:
\bear
\xymatrix{
\check{H}^q(\mathcal{O}_{\mathpzc{SE}})\cong \mathbb{C}^{1|1} \ar[r]^{d^{0}\quad } & \check{H}^q (\Omega^1_{\mathpzc{SE}, odd}) \cong \mathbb{C}^{2|2} \ar[r]^{d^1 \;} &  \check{H}^q (\Omega^2_{\mathpzc{SE}, odd}) \cong \mathbb{C}^{2|2} 
\ar[r]^{\qquad \quad d^2} & \ldots
}
\eear 
where $q=0,1$. Recalling that the maps $d$ in the sequence above are odd, we separately study the two cases.
\begin{enumerate}
\item[$q=0:$] this case corresponds to the map induced by the de Rham differential $d$ on the global sections, which is nothing by $d$ itself. In general, one finds that 
\bear
\check{H}^0 (\mathcal{O}_{\mathpzc{SE}}) \cong \mathbb{C} \cdot \big \{ 1\, | \,\theta \big \}, \qquad \check{H}^{0} (\Omega^{i}_{\mathpzc{SE}, odd}) \cong \mathbb{C} \cdot \big \{ d\theta^i , \, \theta dz d\theta^{i-1}\, |\, dz d\theta^{i-1},\, 
\theta d\theta^i \big \}
\eear
where $\theta$ is a global section of $\Theta_{\mathpzc{E}} \cong \mathcal{O}_{\mathpzc{E}}$ and $i\geq 1$. \\
Now, acting with the de Rham differential on the generators one finds that the first map $d^{0} : \check H^0 (\mathcal{O}_{\mathpzc{SE}}) \rightarrow \check H^0 (\Omega^1_{\mathpzc{SE}, odd})$ is easily identified to be such that 
$\ker d^{0} \cong \mathbb{C}\cdot 1$ and $\mbox{im}\, d^{0} \cong \mathbb{C} \cdot d\theta \subset \check H^0 (\Omega^1_{\mathpzc{SE}, odd})$. The higher maps are such that 
$\ker d^{i} \cong \mathbb{C}\cdot \{ d\theta^i\, | \, dzd\theta^{i-1} \} \subset \check H^0 (\Omega^i_{\mathpzc{SE}, odd})$ and 
$\mbox{im}\, d^{i} \cong \mathbb{C} \cdot \{ d\theta^{i+1}\, | dz d\theta^{i} \} \subset \check H^0 (\Omega^{i+1}_{\mathpzc{SE}, odd})$. In other words the maps are non-zero on the global sections having a $\theta$, for only in this case there is a 
non-vanishing derivative coming from $d$. Using these, one finds that the only non-trivial cohomology groups contributing to $E_2^{i\geq 0, 0}$ are given by
\begin{align}
E^{0,0}_2 \cong \mathbb{C} \cdot 1 \cong \mathbb{C}, &\qquad E_2^{1,0 } \cong \mathbb{C} \cdot dz \cong \Pi \mathbb{C}. 
\end{align}
Notice that these match the non-vanishing cohomology groups for an ordinary elliptic curve, the difference being that those are found at page one, while in the complex supergeometric setting there is an infinite number of non-zero differentials at page one.  
\item[$q=1:$] this case is similar to $q=0,$ upon using the Dolbeault identification $H_{\bar \partial}^{p, q} (\mathpzc{E}) \cong \check{H}^{q} (\Omega^p_{\mathpzc{E}})$, so that under this isomorphism one can systematically multiply the 
above global sections for $q=0$ by $d\bar z \in H^{0,1}_{\bar \partial} (\mathpzc{E})$ to obtain those for $q=1$. Doing this, one finds that the non-trivial cohomology groups contributing to $E_2^{i\geq 0, 1}$ are given by
\begin{align}
E^{0,1}_2 \cong \mathbb{C} \cdot d\bar z \cong \Pi \mathbb{C}, &\qquad E_2^{1,1 } \cong \mathbb{C} \cdot dz d\bar z  \cong  \mathbb{C}, 
\end{align}
and once again one finds an infinite number of non-zero differentials at page one.
\end{enumerate}
Since the differentials at page two read $d : E_2^{p,q} \rightarrow E_2^{p+2, q-1}$, then they are all zero by the above result in cohomology. It follows that  the Hodge-to-de Rham spectral sequence for $\mathpzc{SE}$ converges at page two, \emph{i.e.} 
$E_2 = E_{\infty}$, giving the usual Hodge decomposition of the de Rham cohomology groups.

\end{example}

The above example can be generalized to any super Riemann surface of genus $g\geq 2$: once again one finds that the Hodge-to-de Rham spectral sequence does not converge at page one, but it converge at page two instead. Nonetheless, 
studying the (infinite non-trivial) differentials at page one is not as easy as in the case $g=1$ above. In this case, the dimensions of the cohomology groups that appear at page one have been computed by one of the authors in \cite{CGN}. In any case, 
the above discussion suggests the following 
\begin{problem}
Given a complex supermanifold with K\"ahler associated reduced manifold, does its Hodge-to-de Rham spectral sequence \emph{always} converge at page two? 
\end{problem}
\noindent Finally, notice that in the previous example the case of differential forms can be related to the case of integral forms by the supergeometric analog of \emph{Serre duality}, which indeed involves the Berezinian sheaf in the role of the dualizing sheaf, 
see for example \cite{Penkov}. From this point of view, using differential forms is again equivalent to use integral forms.

{\remark[On Poincar\'e Lemmas in Supergeometry] The previous results rely heavily on the supergeometric generalization of the Poincar\'e Lemma, both for differential forms and integral forms. In the case of differential forms such a generalization 
is completely straightforward and the literature offers various proofs with different level of abstraction of the fact that $\Omega^\bullet_{\mathbb{R}^{p|q}, odd}$ is a right resolution of the constant sheaf $\mathbb{R}$, see for example 
\cite{BR} \cite{Deligne} \cite{Manin}. \\
The story is quite different in the case of integral forms: indeed - to the best knowledge of the authors - a Poincar\'e Lemma for integral forms is stated in \cite{Manin} (as Theorem 3 in chapter 4, paragraph 8) but no proof is provided. As it turns out, the proof of such a theorem is by no means obvious. We fill this gap in the literature by providing a detailed proof. }

\begin{theorem}[Poincar\'e Lemma for Integral Forms] \label{PLInt} Let $\mani$ be a real or complex supermanifold of dimension $p|q$ and let $(\Sigma_\mani^{p-\bullet}, \delta) $ be the complex of integral forms associated to $\mani$. One has
\bear
H^i_{\delta} (\Sigma^{p-\bullet}_\mani) \cong \left \{ \begin{array}{lll}
\mathbb{K}_\mani  & & i = 0 \\
0 & & i \neq 0.
\end{array}
\right.
\eear
In particular, $H^0_{\delta} (\Sigma^{p-\bullet}_\mani)$ is generated by the section $s_0 = \varphi \, \theta_1 \ldots \theta_q \otimes \pi \partial_{x_1} \ldots \pi \partial_{x_p}$, where $\varphi $ is a generating section of the Berezinian sheaf. 
\end{theorem}
\begin{proof} We need to construct a homotopy for the complex. In particular, working locally, we show that for any $k \neq 0$ there exists an homotopy 
$h^k : \mathcal{B}er (\mani) \otimes S^{p - k} \Pi \mathcal{T}_\mani  \rightarrow \mathcal{B}er (\mani) \otimes S^{p-1-k}$ for the differential $\delta$, that is a map such that
$h^{k+1} \circ \delta^{k} + \delta^{k-1} \circ h^k = id_{\mathcal{B}er (\mani) \otimes S^{p-k} \Pi \mathcal{T}_\mani}.$\\
Given a set of local coordinates for $\mani$, we call it $x_a \defeq z_1, \ldots, z_p | \theta_1, \ldots, \theta_q$, and $t \in [0,1]$, we consider the map $(t , x_a ) \stackrel{G}{\longmapsto} t x_a.$ 
This induces a map on sections of the structure sheaf via pull-back, $f(x_a) \stackrel{G^\ast}{\longmapsto} f(tx_a)$. 
We write $G$ as a family of maps parametrized by $t \in [0,1]$, that is $G_t : \mani \rightarrow \mani$, so that we can rewrite the above as a family of pull-back maps $G^\ast_t : \mathcal{O}_\mani \rightarrow \mathcal{O}_\mani$. We define the 
homotopy operator as 
\bear
h^k (\varphi f \otimes F) \defeq (-1)^{|f| + |F|} \varphi \, \sum_b (-1)^{|f| (|x_b| + 1)} \left ( \int_0^1 dt \, t^{Q_s} x_b G^\ast_t f \right ) \otimes \pi \partial_{b} F, 
\eear
where $\varphi$ is a section of the Berezinian, $f $ is a section of the structure sheaf and $F$ is a polyfield of the form $F = \pi \partial^I$ for some multi-index such that $|I| = p-k$ and $Q_s$ is a constant, dependent on the integral form 
$s = \varphi f \otimes F$ and to be determined later on.\\
We now start computing $H \delta (\varphi f \otimes F)$. We have
\begin{align}
H \delta (\varphi f \otimes F) 
& = \varphi \sum_{a,b} (-1)^{(|x_a| + |x_b|)(|f| + |x_a|)} \left ( \int_0^1 dt \, t^{Q_{\delta s}} x_b G^\ast_t ( \partial_a f ) \right ) \otimes \pi \partial_b \cdot  \partial_{\pi \partial_a} F.
\end{align}
Now let us consider $\delta H (\varphi f \otimes F)$. We have
\begin{align}
\delta H (\varphi f \otimes F) 
& = + \varphi \sum_a \int^1_0 dt \, t^{Q_s} G^\ast_t f \otimes F  \label{one} \\
& \quad + \varphi \sum_a (-1)^{|x_a|} \int_0^1 dt \, t^{Q_s} x_b \partial_{ a} G^\ast_t f \otimes F  \label{two} \\
& \quad + \varphi \sum_a (-1)^{|x_a| + 1} \int_0^1 dt \, t^{Q_s} G^\ast_t f \otimes \pi \partial_a \cdot \partial_{\pi \partial_a} F \label{three} \\ 
& \quad - \varphi \sum_{a,b} (-1)^{(|f| + |x_a|) (|x_a| + |x_b|)} \int^1_0 dt \, t^{Q_s} x_b {\partial}_a (G^\ast_t f) \otimes \pi \partial_b \cdot \partial_{\pi \partial_a } F. \label{four}
\end{align}
We see that for the last line \eqref{four} to cancel the term $H\delta (\varphi f \otimes F)$ we need $Q_{\delta s} = Q_{s} + 1,$ by chain-rule.
Let us now study separately the first three lines in the previous expression. Clearly, the first line \eqref{one} yields
\bear
\varphi \sum_a \int^1_0 dt \, t^{Q_s} G^\ast_t f \otimes F = (p+q)  \varphi \left (\int^1_0 dt \, t^{Q_s} G^\ast_t f \right ) \otimes F.  
\eear
Let us now look at the second line \eqref{two}. Without loss of generality we can assume that $f$ is homogeneous of degree $\deg_{\theta} (f)$ in the theta's, so that we can rewrite
\begin{align}
\sum_a (-1)^{|x_a|}x_a \partial_a (G^\ast_t f ) & = \sum_{i = 1}^p z_i \partial_{z_i} f (t z| t \theta) - \sum_{\alpha =1}^q \theta_\alpha \partial_{\theta_\alpha} f (t  z | t \theta) \nonumber \\
&=  t \frac{d}{dt} f(t x) - 2 \deg_{\theta} (f)  f(tx).
\end{align}
It follows that the equation \eqref{two} can be computed as
\begin{align}
\varphi \sum_a (-1)^{|x_a|} \int_0^1 dt \, t^{Q_s} x_b \partial_{ a} G^\ast_t f \otimes F  & = \varphi \int_{0}^1dt \, t^{Q_s} \left ( t \frac{d}{dt} f(t x) - 2 \deg_{\theta} (f)   f(tx) \right ) \otimes F \nonumber \\
& = \varphi f \otimes F - \delta_{Q_s + 1 + \deg_{\theta} (f), 0 } \, ( \varphi f (0) \otimes F )+ \nonumber \\ 
& \quad - \left ( Q_s + 1 + 2 \deg_{\theta} (f) \right ) \varphi \left ( \int_0^1 dt\, t^{Q_s} G^\ast_t f \right ) \otimes F,
\end{align}
by integration by parts.
Finally, denoting $\deg_{\pi \partial_\theta} (F)$ and $\deg_{\pi \partial_z} (F)$ the degree of $F$ in the even ($\pi \partial_{\theta}$) and odd ($\pi \partial_z$) monomials of the polyfield $F$ respectively, it can be observed that
\bear
\sum_a (-1)^{|x_a| + 1}\pi \partial_a \partial_{\pi \partial_a } F = \left (\deg_{\pi \partial_\theta} (F) - \deg_{\pi \partial_z} (F) \right )  F,
\eear
so that the third line \eqref{three} can be rewritten as
\begin{align}
\varphi \sum_a (-1)^{|x_a| + 1} \int_0^1 dt \, t^{Q_s} G^\ast_t f \otimes \pi \partial_a \cdot \partial_{\pi \partial_a} 
= \left (\deg_{\pi \partial_\theta} (F) - \deg_{\pi \partial_z} (F) \right ) \varphi \left ( \int_0^1 dt\, t^{Q_s } G^{\ast}_t f \right ) \otimes F.  
\end{align}
Gathering together all the contributions one has 
\begin{align}
(\delta H + H \delta )(\varphi f \otimes F) &  = \varphi f \otimes F - \delta_{Q_s +1 + \deg_{\theta} (f), 0} \,\varphi \, f(0) \otimes F + \nonumber \\
& + \left (p+q + \deg_{\pi \partial_\theta} (F) - \deg_{\pi \partial_z } (F) - 2\deg_{\theta} (f) - Q_s - 1 \right ) \varphi \int_0^1 dt\, t^Q G^{\ast}_t f \otimes F.
\end{align}
Let us now look at the condition on $Q_s$ in order to have an homotopy. We have to require that 
\bear
Q_s = p+q + \deg_{\pi \partial_\theta} (F) - \deg_{\pi \partial_z } (F) - 2 \deg_{\theta} (f) -1.
\eear
This in turn leads to
\bear
(\delta H + H \delta )(\varphi f \otimes F) &  = \varphi f \otimes F - \delta_{(p+q + \deg_{\pi \partial_\theta} (F) - \deg_{\pi \partial_z } (F) - \deg_{\theta} (f)) , 0} \,\varphi \, f(0|\theta) \otimes F.
\eear
We now note that $\deg_{\pi \partial_\theta} (F) \geq 0$, $0\leq \deg_{\pi \partial_z } (F) \leq p$ and $0\leq \deg_{\theta} (f) \leq q$, therefore the only instance in which the above fails to be a homotopy corresponds to the choices 
\bear
\left \{ \begin{array}{l}
\deg_{\pi \partial_\theta} (F) = 0 \\
\deg_{\pi \partial_z } (F) = p \\
\deg_{\theta} (f) = q,
\end{array}
\right.
\eear
which in turn lead to the following generator for the only non-trivial cohomology group 
\bear
H^0_\delta (\Sigma^{p - \bullet}) = k \cdot \left ( \varphi\, \theta_1 \ldots \theta_q \otimes \pi \partial_{z_1} \ldots \pi \partial_{z_p} \right ),
\eear
for $k\in \mathbb{K}$. Note that by the very definition of $\delta$ this element is indeed closed and not exact, since it has the maximal amount of both theta's and odd sections $\pi \partial_{z}$'s, thus concluding the proof.
\end{proof}
{\remark[On Algebraic Versus Real or Complex Supermanifolds] With reference to the spectral sequences $E^\Omega_r$ and $E^\Sigma_r$ related to the double complex $_\mathpzc{D}\mathcal{V}^{\bullet \bullet}_\mani$, it is worth to remark that both 
the homotopy operators of the universal de Rham and Spencer complex are algebraic, as no integral is involved. It follows that {at page one} the results for $E^\Omega_r$ and $E^\Sigma_r$ hold true also for algebraic supermanifolds and, more 
in general, for superschemes, so that one can indeed recover differential and integral forms from the virtual forms double complex also when working in the algebraic category. \\
This is no longer true already {at page two}: indeed both the homotopy operators appearing in the Poincar\'e Lemmas for differential and integral forms require an integration - for the case of integral forms, see above in the proof of Theorem \ref{PLInt}. 
It follows that the related results holds true only in the smooth and analytic category, but break down in the algebraic category. Notice by the way that such a difficulty also exists in the ordinary commutative setting. }

\appendix

\section{Lie Derivative on $(\Omega^{\bullet}_{\mani, odd})^\ast$}

\noindent We prove the properties of the Lie derivative on $(\Omega^{\bullet}_{\mani, odd})^{\ast}$ defined in \ref{LieStar}, as stated in Lemma \ref{PropLie} which we repeat here for the sake of readability. 
\begin{lemma}  The Lie derivative $\mathfrak{L}_X : (\Omega^{\bullet}_{\mani, odd})^\ast \rightarrow (\Omega^{\bullet}_{\mani, odd})^\ast$ has the following properties: 
\begin{enumerate}
\item $\mathfrak{L}_X (\tau) = \pi [X , \pi \tau]$ for any $\tau \in \Pi \mathcal{T}_\mani$;
\item $\mathfrak{L}_X$ is a superderivation of $(\Omega^\bullet_{\mani, odd})^{\ast}$, \emph{i.e.}\ the super Leibniz rule holds true:
\bear \label{superd1}
\mathfrak{L}_X  (\tau_1 \tau_2) = \mathfrak{L}_X (\tau_1) \tau_2 + (-1)^{|X||\tau_1|} \tau_1 \mathfrak{L}_{X} (\tau_2)
\eear
for any $\tau_1, \tau_2 \in (\Omega^{\bullet}_{\mani, odd})^\ast$ and $X \in \mathcal{T}_\mani;$
\item $\mathfrak{L}_{fX} (\tau) = f \mathfrak{L}_{X} (\tau) + (-1)^{|X||f|} \pi X \langle df, \tau \rangle$ for any $f \in \mathcal{O}_\mani$, $\tau \in (\Omega^{\bullet}_{\mani, odd})^\ast$.
\end{enumerate}
\end{lemma}
\begin{proof} We prove the claims separately.
\begin{enumerate}
\item We write $\tau = \sum_a g_a \pi \partial_a$ and $X = \sum_b f_b \partial_b$ and we set $\mathfrak{L}_X (\tau) \defeq \sum_c h_c \pi \partial_c$. Now, noticing that 
$
\langle \tau , dx_a \rangle =g_a (-1)^{(|x_a| + 1)(|x_b| + 1)} \langle dx_a, \pi \partial_b \rangle = g_a, 
$ we compute
\begin{align}
\mathfrak{L}_X (g_a) & = \mathfrak{L}_X (\langle \tau , dx_a \rangle ) = \langle \mathfrak{L}_X (\tau), dx_a \rangle + (-1)^{|\tau||X|} \langle \tau, \mathcal{L}_X (dx_a) \rangle  \nonumber \\
& = h_a + (-1)^{|X| (|\tau| + 1)} \sum_b g_b (\partial_b f_a).
\end{align}
It follows that $h_a = \sum_b f_b (\partial_b g_a) - (-1)^{|X| |\pi \tau|}\sum_b g_b (\partial_b f_a),$
hence $\mathfrak{L}_X (\tau)= \pi ([X , \pi \tau]).$
\item We now prove that $\mathfrak{L}_X $ is a superderivation, showing that the \eqref{superd1} holds true by double induction on the degrees $( \mbox{deg} (\tau_1), \mbox{deg} (\tau_2))$ in $Sym^\bullet \Pi \mathcal{T}_\mani$. 
Clearly, the cases $(0,1)$ and $(1, 0)$ are guaranteed by the previous point in the proof. Next, since for a 1-form $\omega \in \Omega^1_{\mani, odd} $ one has 
\bear
\langle \tau_1 \tau_2 , \omega \rangle = \tau_1 \langle \tau_2, \omega \rangle + (-1)^{|\omega| |\tau_2|} \langle \tau_1, \omega \rangle \tau_2, 
\eear
then 
\bear
\mathfrak{L}_X (\langle \tau_1 \tau_2 , \omega \rangle ) = \mathfrak{L}_X (\tau_1 \langle \tau_2 , \omega \rangle ) + (-1)^{|\omega| |\tau_2|} \mathfrak{L}_X (\langle \tau_1 , \omega \rangle \tau_2), 
\eear
which, by inductive hypothesis is equal to 
\begin{align}
\mathfrak{L}_X (\langle \tau_1 \tau_2 , \omega \rangle ) 
& = \langle \mathfrak{L}_X (\tau_1) \tau_2 , \omega \rangle + (-1)^{|X|(|\tau_1| + |\tau_2|)} \langle \tau_1 \tau_2 , \mathcal{L}_X (\omega) \rangle + (-1)^{|X| |\tau_1|} \langle \tau_1 \mathfrak{L}_{X} (\tau_2) , \omega \rangle.   \label{I}
\end{align}
On the other hand one has
\bear \label{II}
\mathfrak{L}_{X} (\langle \tau_1 \tau_2 , \omega \rangle ) = \langle \mathfrak{L}_X (\tau_1 \tau_2 ), \omega \rangle + (-1)^{|X| (|\tau_1| + |\tau_2|)} \langle \tau_1 \tau_2 , \mathcal{L}_X (\omega) \rangle, 
\eear
so that comparing \eqref{I} with \eqref{II} one get the super Leibniz rule.
\item Let us first prove the case $\tau \in \Pi \mathcal{T}_\mani$, using the first point of the lemma. One has
\begin{align} 
\mathfrak{L}_{f X} (\tau) & = \pi [f X, \pi \tau] 
 = f \mathfrak{L}_X (\tau) + (-1)^{|X||f| } \pi X \langle df , \tau \rangle .
\end{align}
In order to conclude the proof, one can observe that $\mathfrak{L}_{fX}$ and $f\mathfrak{L}_X$ are both left derivations of $Sym^\bullet \Pi \mathcal{T}_\mani$ for $X$ and $f$ fixed. The same holds true for $\pi X \langle df , \, \cdot \, \rangle $, indeed
\begin{align}
\pi X \langle df , \tau_1 \tau_2 \rangle 
& =   \pi X \langle df , \tau_1 \rangle \tau_2 + (-1)^{|\tau_1| (|\pi X| + |df|)} \tau_1 \pi X \langle df, \tau_2 \rangle.
\end{align}
Then, thanks to the super Leibniz rule, the property proved above for $\tau \in \Pi \mathcal{T}_\mani$ holds true for any $\tau \in Sym^\bullet \Pi \mathcal{T}_\mani.$
\end{enumerate}
\end{proof}

\end{document}